\documentclass[12pt]{amsart}
\usepackage{latexsym}
\usepackage[T2A,T1]{fontenc}
\usepackage[utf8]{inputenc}
\usepackage[english]{babel}
\usepackage{a4}
\usepackage[bookmarks]{hyperref}
\usepackage{amsmath}
\usepackage{amssymb}
\usepackage{amsthm}
\usepackage{mathtools}
\usepackage{bm}%
\usepackage{xcolor}
\usepackage{tikz}
\usetikzlibrary{shapes,calc,patterns}
\usepackage{url}
\usepackage{enumerate}
\usepackage{mdwlist}%
\usepackage{braket}
\usepackage{mleftright}%
\usepackage{array}
\usepackage{tabularx}%
\usepackage{booktabs}%

\theoremstyle{definition}
\newtheorem{definition}{Definition}[section]
\theoremstyle{plain}
\newtheorem{theorem}[definition]{Theorem}
\theoremstyle{plain}
\newtheorem{proposition}[definition]{Proposition}
\theoremstyle{plain}
\newtheorem{lemma}[definition]{Lemma}
\theoremstyle{plain}
\newtheorem{corollary}[definition]{Corollary}
\theoremstyle{remark}
\newtheorem{claim}{Claim}
\theoremstyle{definition}
\newtheorem{problem}{Problem}
\numberwithin{equation}{section}

\DeclareMathOperator{\Inv}{\mathsf{Inv}}%
\DeclareMathOperator{\POL}{Pol}%
\DeclareMathOperator{\Con}{Con}%
\newcommand{\GF}[1]{\mathsf{GF}(#1)}%
\newcommand{\ab}[1]{{\mathbf{#1}}}%
\newcommand{\Mtct}[3]{\mathtt{M}_{\ab{#1}}(#2,#3)}%
\newcommand{\quotienttct}[2]{\Braket{#1,#2}}
\newcommand{\interval}[2]{\mathtt{I}[#1,#2]}%

\newcommand{\bottom}[1]{0_{#1}}%
\newcommand{\uno}[1]{1_{#1}}%
\newcommand{\Cg}[2][\ab{A}]{\text{Cg}_{#1}(#2)}%
\newcommand{\numberset}{\mathbb}
\newcommand{\N}{\numberset{N}}

\newcommand{\Z}{\numberset{Z}}
\newcommand{\finset}[1]{\{1,\dots, #1\}}%
\DeclareMathAlphabet{\mathbfsl}{OT1}{cmr}{bx}{it}

\renewcommand{\vec}[1]{{\boldsymbol{#1}}}%
\newcommand{\ari}[1]{_{#1}}%
\newcommand{\strset}[1]{\mathcal{#1}}%
\newcommand{\restrict}[1]{\rvert_{#1}}%
\newcommand{\card}[1]{\mleft\lvert#1\mright\rvert}%

\title{On polynomial completeness properties of finite Mal'cev algebras}
\author{Bernardo Rossi}
\address{
Institut für Algebra,
Johannes Kepler Universität Linz, Altenberger Straße 69, 4040 Linz, Austria.}
\address{Department of Algebra, Faculty of Mathematics and Physics, 
Charles University, Sokolovsk\'{a} 49/83 186 75, Praha 8, Czechia.}
\email{bernardo.rossi@matfyz.cuni.cz}
\thanks{Supported by the Austrian Science Fund (FWF):~P33878 and by
Charles University under grant no. PRIMUS/24/SCI/008.}
\begin{document}
\begin{abstract}
 Polynomial completeness results aim at
 characterizing those functions that are 
 induced by polynomials. 
 Each polynomial function is congruence preserving, but the
 opposite need not be true.
 A finite algebraic structure $\ab{A}$ is called 
 strictly 1-affine complete if every 
 unary partial function from a subset of $A$ to $A$
 that preserves 
 the congruences of $\ab{A}$ can be interpolated by a 
 polynomial function of $\ab{A}$. 
 The problem of characterizing strictly 1-affine complete 
 finite Mal'cev algebras
 is still open. 
 In this paper we extend the characterization by 
 E. Aichinger and P. Idziak of 
 strictly 1-affine complete expanded groups to
 finite congruence regular Mal'cev algebras. 
\end{abstract}

\date{\today}

\maketitle
\section{Introduction}
Inspired by the work of G. Grätzer on the variety of Boolean algebras 
\cite{Gra62}
and by the work of R. Wille on congruence class geometries 
\cite{Wil70},
in the past century 
several authors studied
those algebraic structures that
exhibit some form of completeness with respect to polynomial
or term functions 
\cite{Wer74, Wer78, Gum79, PilSo80, Nob85, KaaPix87}.
For an algebraic structure $\ab{A}$
a $k$-ary polynomial of $\ab{A}$ is an expression of
the form 
$t(x_1 , \dots , x_k , a_1 , \dots , a_m )$, where $t$ is a $(k+m)$-ary 
term in the language of $\ab{A}$
and $a_1, \dots, a_m$ are elements of the universe of $\ab{A}$. 
For a finitary operation $f$ on $A$ of arity $k$, 
if for all 
$\vec{y}\in A^k$ we have 
$f(\vec{y})=t^{\ab{A}}(y_1 , \dots , y_k , a_1 , \dots , a_m)$,
we say that $f$ is \emph{induced} by the polynomial 
$t(x_1 , \dots , x_k , a_1 , \dots , a_m )$.
Functions induced by polynomials are usually called \emph{polynomial functions}. 
Polynomial functions have several properties,
in particular they preserve congruences. On the other hand, 
in general, not all congruence-preserving functions are induced
by polynomials. For instance, the map $x\mapsto x^2$ 
preserves the congruences of the group $\Z_4$
but it
is not a
polynomial function.
An algebra $\ab{A}$ whose polynomial functions are 
exactly those 
functions that preserve the congruences of $\ab{A}$ is called 
\emph{affine complete}.
In \cite[Problem~6]{Gra08}
the author asks for 
a characterization of affine complete algebras.
The problem
in its full generality remains open; 
nonetheless, many partial results have been proved.
In 1971, H. Werner proved that a vector space is 
affine complete if and only if its dimension 
is not 1 \cite{Wer71}. 
In 1976, W. Nöbauer characterized
all affine complete finitely generated 
%modules
abelian groups
\cite{Nob76}. 
All unary congruence-preserving functions over finite abelian groups 
were determined in \cite{LauNob76}, and the 
ones over symmetric groups in \cite{Kai77}.
A parallel line of research focuses on
affine complete varieties, that is varieties that
consist entirely of affine complete algebras. 
In \cite{KaaMcK97} it was proved that
each such variety must be congruence distributive,
\cite[Theorem~4.1]{KaaMcK97};
and a characterization was given of 
affine complete arithmetical varieties of countable 
type \cite[Corollary~4.4]{KaaMcK97}. 
For a comprehensive survey of the topic we refer the 
reader to the monograph~\cite{KaaPix01} and to the 
paper~\cite{AicPil03}. 
If each $k$-ary operation on $A$ that 
preserves the congruences of $\ab{A}$ 
is a polynomial function of $\ab{A}$, 
then $\ab{A}$ is called
\emph{$k$-affine complete}.
In~\cite[Corollary~11.6]{AicMud09}, E. Aichinger and N. Mudrinski 
proved a characterization of $k$-affine complete 
finite algebras with a group reduct 
(also known in the literature as \emph{expanded groups}) 
whose lattice 
of ideals satisfies a property called (APMI).

In this paper we focus on a slightly different question:
Let $\ab{A}$ be 
an algebra; 
we say that
$\ab{A}$ is \emph{(locally) strictly $k$-affine complete}
if each $k$-ary partial function defined on a (finite) 
subset of $A^k$ that preserves 
the congruences of $\ab{A}$ can be 
interpolated on its domain by a $k$-ary polynomial 
function
of $\ab{A}$. 
An algebra that is (locally) strictly $k$-affine complete 
for each $k\in \N$ is called \emph{(locally)
strictly affine complete}. 
In~\cite{HagHer82} (cf.~\cite[Section~5.1]{KaaPix01} and 
\cite[Proposition~5.2]{Aic00})
it is proved that an algebra in a congruence permutable variety 
is locally 
strictly affine complete if and only if it is congruence 
neutral, a notion coming from commutator theory 
(cf.~\cite{FreMcK87} and \cite[Section~4.13]{McKMcnTay88}).
Moreover, in \cite{HagHer82, Aic00}, it is proved that
for $k\geq 2$
an algebra in a congruence permutable variety
is locally strictly $k$-affine complete
if and only if it is locally strictly $2$-affine complete. 
Clearly, a finite algebra is locally strictly $k$-affine
complete if and only if it is strictly $k$-affine
complete.  
These results leave the following problem open: 
\begin{problem}\label{problem:char_s1acMalcev}
Characterize the finite algebras in a congruence permutable 
variety that are strictly $1$-affine complete. 
\end{problem}
A partial answer to Problem~\ref{problem:char_s1acMalcev}
was given in \cite{AicIdz04},
where the authors gave a characterization 
of strictly $1$-affine complete finite expanded groups.
In this paper we extend their 
characterization to the class of finite congruence regular algebras 
with a Mal'cev polynomial.

In the present note we will use the
standard notation of universal algebra,
clone theory 
and tame congruence theory
(cf.~respectively~\cite[Chapters~1-4]{McKMcnTay88},
\cite[Chapter~E]{PosKal79},
and \cite[Chapters~1-4]{HobbMcK}).
For an algebra $\ab{A}$ we will call 
the partial polymorphisms of the set $\Con\ab{A}$
\emph{partial compatible functions} 
or \emph{partial congruence-preserving functions}.
By $[\cdot, \cdot]$, we
will denote the binary term condition commutator 
as defined in \cite[Definition~4.150]{McKMcnTay88}.
For a pair of congruences $\alpha, \beta$ of an algebra
$\ab{A}$
in a congruence modular variety,
we let $(\alpha :\beta)$ be the largest congruence 
$\eta$ such that $[\eta, \beta]\leq \alpha$. 
This congruence is called the 
\emph{centralizer of $\alpha$ and $\beta$}
(for algebras in congruence modular varieties
this definition is equivalent 
to~\cite[Definition~4.150]{McKMcnTay88}  
by \cite[Proposition~4.2]{FreMcK87}).
We
will sometimes say that an interval 
$\interval{\alpha}{\beta}$ of $\Con\ab{A}$
is abelian if the quotient $\quotienttct{\alpha}{\beta}$
is abelian according to~\cite[Definition~3.6(6)]{HobbMcK}.
In contrast to the standard definition, we will
refer to an algebra with a Mal'cev polynomial
as a \emph{Mal'cev algebra}. 
Note that in \cite{BurSan81, McKMcnTay88} a Mal'cev algebra 
is an algebra that generates a congruence permutable variety. 
Thus, all algebras that are Mal'cev according to the 
standard definition satisfy our definition, 
while the opposite is not true, as witnessed by
the algebra $(\{0,1\}; \rightarrow)$.
In \cite{IdzSlo01} the authors developed a theory of Mal'cev algebras 
whose congruence lattice satisfies a property 
called (SC1) that plays a central role in our characterization.
\begin{definition}\label{def:SC1}
A Mal'cev algebra $\ab{A}$ satisfies the condition (SC1) 
if for every strictly meet irreducible congruence $\mu$ of $\ab{A}$
we have $(\mu:\mu^+)\leq \mu^+$. 
\end{definition}
Following \cite{AicIdz04}, we introduce the following 
definition:
\begin{definition}\label{def:definition_ABp}
Let $\ab{A}$ be a Mal'cev algebra, let
$p$ be a prime number, and let $\gamma\in \Con\ab{A}$.
We say that $(\ab{A}, \gamma)$ 
has property (AB$p$) if for all $\alpha, \beta\in \interval{\bottom{A},\gamma}$
with $\alpha\prec\beta$ and $[\beta, \beta]\leq \alpha$, and 
for each $a\in A$ we have 
that $\card{(a/\alpha)/(\beta/\alpha)}\in\{1, p\} $.
In other words; for each abelian prime quotient 
$\quotienttct{\alpha}{\beta}$
in the interval $\interval{\bottom{A}}{\gamma}$
each $\beta$-class is the union of either $1$ or $p$ distinct
$\alpha$-classes. 
We say that $\ab{A}$ has property (AB$p$) if $(\ab{A}, \uno{A})$
has property (AB$p$).  
\end{definition}

In \cite{AicIdz04} the authors characterize the 
strictly 1-affine complete Mal'cev algebras with
a group reduct as those that satisfy (SC1) and (AB2)
(cf.~\cite[Theorem~1.3]{AicIdz04}).
The main result of this paper is a generalization
from finite expanded groups to finite Mal'cev algebras
of  the implication ``(1) $\Rightarrow$ (2)'' of \cite[Theorem~1.3]{AicIdz04}:
\begin{theorem}\label{teor.main_theorem_sc1_ab2_imply_compl}
Let $\ab{A}$ be finite Mal'cev algebra with (SC1) and (AB2).
Then $\ab{A}$ is strictly 1-affine complete. 
\end{theorem}
A proof will be given at the end of Section~\ref{sec:proof_theorem_strict_affine_compl}.
In Section~\ref{sec:concluding_remarks} we will use 
Theorem~\ref{teor.main_theorem_sc1_ab2_imply_compl}
to characterize strictly $1$-affine complete congruence regular
Mal'cev algebras:
\begin{theorem}\label{theorem:Malcev_and_regular_characterization_s1ac}
Let $\ab{A}$ be a finite congruence regular Mal'cev algebra. 
Then the following are equivalent:
\begin{enumerate}
\item $\ab{A}$ is strictly 1-affine complete;\label{item:s1ac:theorem:MMalcev_and_regular_characterization_s1ac}
\item $\ab{A}$ satisfies (SC1) and (AB2).\label{item:ab2and sc1:theorem:Malcev_and_regular_characterization_s1ac}
\end{enumerate}
\end{theorem}
Moreover, we will explain how 
Theorem~\ref{teor.main_theorem_sc1_ab2_imply_compl}
is related to 
a possible solution of Problem~\ref{problem:char_s1acMalcev}.
Finally, in Corollary~\ref{cor:loops} we specify our results to
the case of finite loops, extending 
\cite[Corollary~11.2]{AicIdz04} from finite groups to finite loops.

\section{Abelian congruences in Mal'cev algebras}\label{sec:preliminary_on_abelian_congruences}
In this section we recall some basic facts on abelian 
congruences in Mal'cev algebras. 
The fact that every abelian algebra in a congruence modular 
variety is affine was first
proved in \cite{Her79}. 
For more details on abelian algebras in congruence 
modular varieties we refer the reader to \cite[Chapter~5]{FreMcK87}. 
The coordinatization of abelian congruences that we report can also 
be found in \cite{Fre83, Aic18}.
We start with two preliminary results about commutators and 
centralizers in Mal'cev algebras. 
\begin{lemma}[{cf.~\cite[Proposition~2.6]{Aic06}}]\label{lemma:ex_citation_to_prop2.6Aic06}
Let $k\in\N$, let~$\ab{A}$ be an algebra
with a Mal'cev
polynomial~$d$, let $\alpha,\beta\in\Con\ab{A}$, and let
$p\in\POL\ari{k}\ab{A}$. For all $\vec{u},\vec{v},\vec{w}\in A^k$ such
that $\vec{u}\equiv_\alpha\vec{v}\equiv_\beta\vec{w}$, we have
\[
d(p(\vec{u}),p(\vec{v}),p(\vec{w}))
\equiv_{[\alpha,\beta]}
p(d(u_1,v_1,w_1),d(u_2,v_2,w_2),\dotsc,d(u_k,v_k,w_k)).
\]
\end{lemma}
\begin{proposition}[{cf.~\cite[Proposition~4.3]{FreMcK87}}]\label{prop:commutator_lattice}
Let $\ab{A}$ be a Mal'cev algebra, and let $[\cdot, \cdot  ]$ 
be the term condition commutator operation on $\Con \ab{A}$ 
as defined in \cite[Definition~4.150]{McKMcnTay88}. 
Then $(\Con \ab{A}; \wedge, \vee, [\cdot, \cdot])$ is 
a commutator lattice as defined in \cite[Definition~3.1]{Aic18}.
Moreover, the centralizer operation $(\cdot : \cdot)$,
as defined in \cite[Definition~4.150]{McKMcnTay88}, 
coincides with the residuation defined in \cite[Section~3]{Aic18}.
\end{proposition}
Let $\ab{A}$ be a Mal'cev algebra with Mal'cev 
polynomial $d$, and let $o\in A$.
Then we define two binary operations on $A$ as follows:
For all $x_1, x_2\in A$ we let
\begin{equation}\label{eq:equazione_che_definisce_il_gruppo_abeliano_del_termine_di_Malcev}
\begin{split}
x_1+_ox_2&:=d(x_1, o, x_2)\text{ and }\\
x_1 -_o x_2&:=d(x_1, x_2, o).
\end{split}
\end{equation}
For $x\in A$, we let $-_o x:=o-_ox$.
The following lemma goes back to \cite{Her79,Fre83}; a proof that relies only
on the definition of the commutator operation can be found in \cite[Section~3]{Aic19a}. 
\begin{lemma}\label{lemma:piu_emeno_fanno_gruppoo_abeliano_nella_classe_congruenza}
Let $\ab{A}$ be a Mal'cev algebra, let $\alpha\in \Con\ab{A}$ with
$[\alpha, \alpha]=\bottom{A}$, let $o\in A$, and let $+_o$ and $-_o$
be defined as in \eqref{eq:equazione_che_definisce_il_gruppo_abeliano_del_termine_di_Malcev}.
Then $(o/\alpha; +_o, -_o, o)$ is an abelian group.  
\end{lemma}
Let $\ab{A}$ be a Mal'cev algebra, let $\alpha\in \Con\ab{A}$ with
$[\alpha, \alpha]=\bottom{A}$, let $o\in A$, let $+_o$ and $-_o$
be defined as in \eqref{eq:equazione_che_definisce_il_gruppo_abeliano_del_termine_di_Malcev}.
We define
\begin{equation}\label{eq:equazione_che_definisce_l'universo_dell'anello_dei_polinomi_ristretti}
R_o=\{p\restrict{o/\alpha}\mid p\in\POL\ari{1}\ab{A}\text{ and }p(o)=o\}.
\end{equation}
Furthermore, on $R_o$ we define two operations as follows:
For all $p,q\in R_o$ we let 
\begin{equation}\label{eq:equazione_che_definisce_le_operazioni_dell'anello_dei_polinomi_ristretti}
\begin{split}
(p+q)(x)&:=p(x)+_o q(x)=d(p(x),o,q(x))\text{ and } \\
(p\circ q)(x)&:=p(q(x)), \text{ for all } x\in A.
\end{split}
\end{equation} 
We define $\ab{R}_o:=(R_o; +, \circ)$,
and its action on
$o/\alpha$ as follows: For all $r\in R_o$ and $x\in o/\alpha$ we let 
$r\cdot x:=r(x)$.  
The next lemma follows from \cite{Aic19a}. 
\begin{lemma}\label{lemma:piu_meno_fanno_modulo_su-anello_polinomi_ristretto_nella_classe_congruenza}
Let $\ab{A}$ be a Mal'cev algebra, let $\alpha\in \Con\ab{A}$ with
$[\alpha, \alpha]=\bottom{A}$, let $o\in A$, let $+_o$ and $-_o$
be defined as in \eqref{eq:equazione_che_definisce_il_gruppo_abeliano_del_termine_di_Malcev},
and let $\ab{R}_o$ be defined as in
\eqref{eq:equazione_che_definisce_l'universo_dell'anello_dei_polinomi_ristretti}
and \eqref{eq:equazione_che_definisce_le_operazioni_dell'anello_dei_polinomi_ristretti}. Then 
$\ab{R}_o$ is a ring with unity and $(o/\alpha; +_o)$ is a module over $\ab{R}_o$. 
Moreover, the algebra $\ab{A}\restrict{o/\alpha}$ is polynomially equivalent to 
the $\ab{R}_o$-module $(o/\alpha; +_o)$. 
\end{lemma}

\section{Projective intervals in Mal'cev algebras}\label{sec:projective_intervals_in_Malcev}
In this section we investigate some properties of projective intervals 
in Mal'cev algebras. 
Given a lattice $\ab{L}$ and two intervals 
$\interval{\alpha}{\beta}$ and $\interval{\gamma}{\delta}$
we say that $\interval{\alpha}{\beta}$ 
\emph{transposes up to} $\interval{\gamma}{\delta}$,
and we write $\interval{\alpha}{\beta}\nearrow\interval{\gamma}{\delta}$,
if $\beta\vee \gamma= \delta$ and $\beta\wedge \gamma=\alpha$. 
Analogously, we say that $\interval{\alpha}{\beta}$ 
\emph{transposes down to} $\interval{\gamma}{\delta}$,
and we write $\interval{\alpha}{\beta}\searrow\interval{\gamma}{\delta}$,
if $\beta=\alpha\vee \delta$ and $\gamma=\alpha\wedge \delta$. 
We say that $\interval{\alpha}{\beta}$
and $\interval{\gamma}{\delta}$ are 
\emph{projective}, and we write 
$\interval{\alpha}{\beta}\leftrightsquigarrow\interval{\gamma}{\delta}$,
if there is a finite sequence 
\[\interval{\alpha}{\beta}=\interval{\eta_0}{\theta_0},
\interval{\eta_1}{\theta_1},\dots, 
\interval{\eta_n}{\theta_n}=\interval{\gamma}{\delta}
\]
such that $\interval{\eta_i}{\theta_i}$ transposes up or down 
to $\interval{\eta_{i+1}}{\theta_{i+1}}$ for each $i<n$. 
For 
further results
about projective intervals in modular lattices we refer the reader 
to \cite[Chapter~2]{McKMcnTay88}.
\begin{lemma}\label{lemma:exAic18Lemma3.4}
Let $\ab{A}$ be a Mal'cev algebra, let $\alpha, \beta, \gamma,\delta\in \Con\ab{A}$ 
with $\alpha\leq\beta$, $\gamma\leq\delta$, 
and $\interval{\alpha}{\beta}\leftrightsquigarrow\interval{\gamma}{\delta}$. 
Then we have:
\begin{enumerate}
\item $(\alpha:\beta)=(\gamma:\delta)$;\label{item:centralizerinprojectiveintervals}
\item $[\beta, \beta]\leq \alpha$ if and only if $[\delta, \delta]\leq \gamma$. \label{item:typesinprojactiveintervals}
\end{enumerate}
\end{lemma}
\begin{proof}
Proposition~\ref{prop:commutator_lattice} implies that the assumptions 
of \cite[Lemma~3.4]{Aic18} are satisfied, and the statement follows. 
\end{proof}
The following
three lemmata can be viewed as partial generalization 
of the fact that in a group $\ab{G}$ given two normal 
subgroups $A$ and $B$, we have $A/(A\cap B)\cong (A+B)/B$. 
\begin{lemma}\label{lemma:conseguenze_permutabilita_intervalli_proiettivi_esistenza,d(b,o,c)}
Let $\ab{A}$ be a Mal'cev algebra, let $\alpha,\beta,\gamma,\delta\in \Con\ab{A}$
with $\alpha\leq\beta$, $\gamma\leq\delta$ and 
$\interval{\alpha}{\beta}\nearrow\interval{\gamma}{\delta}$. 
Then for all $o\in A$ and for each $x\in o/\delta$ there exist 
$b\in o/\beta$ and $c\in o/\gamma$
such that $x\mathrel{\alpha} d(b,o,c)$. 
\end{lemma}
\begin{proof}
Let $o\in A$ and let $x\in o/\delta$.
Since $\delta=\beta\vee \gamma$ and $\ab{A}$ is congruence permutable,
there exists $y\in A$ such that
$o\mathrel{\gamma}y\mathrel{\beta}x$. 
Thus, the fact that
$\beta\wedge\gamma=\alpha$
implies that $((x-_oy)+_oy)\mathrel{\alpha}x$.
Thus, since $x-_o y\mathrel{\beta}d(x,x,o)=o$ and $y\mathrel{\gamma}o$,
we have that $b:=d(x,y,o)\in o/\beta$, $c:=y\in o/\gamma$ and 
$x\mathrel{\alpha}d(b,o,c)$. 
\end{proof}
\begin{lemma}\label{lemma:ismorphisms_between_classes_same_element_projective_intervals}
Let $\ab{A}$ be a Mal'cev algebra,
let $\alpha,\beta,\gamma,\delta\in \Con\ab{A}$
with $\alpha\leq\beta$, $\gamma\leq\delta$ and 
$\interval{\alpha}{\beta}\nearrow\interval{\gamma}{\delta}$,
let $o\in A$ and let  
\[
h_o:=\{(x/\gamma,b/\alpha)\in (o/\gamma)/(\delta/\gamma)\times
(o/\alpha)/(\beta/\alpha)\mid \exists c\in o/\gamma\colon x\mathrel{\alpha}d(b,o,c)\}.
\] 
Then $h_o$ is a bijection between 
$(o/\gamma)/(\delta/\gamma)$ and 
$(o/\alpha)/(\beta/\alpha)$. 
\end{lemma}
\begin{proof}
First, we show that $h_o$ is functional. 
To this end, let $x\in o/\delta$ and let $b_1, b_2\in o/\beta$ 
and $c_1, c_2\in o/\gamma$ be such that
$d(b_1,o,c_1)\mathrel{\alpha}x\mathrel{\alpha}d(b_2, o, c_2)$.
Note that the existence of at least one such pair $b\in o/\beta$ 
and $c\in o/\gamma$ 
is a consequence 
of~Lemma~\ref{lemma:conseguenze_permutabilita_intervalli_proiettivi_esistenza,d(b,o,c)}. 
Then we have that 
\[
b_1=d(b_1,o,o)\mathrel{\gamma}d(b_1, o,c_1)
\mathrel{\alpha}d(b_2,o,c_2)\mathrel{\gamma}d(b_2, o,o)=b_2.
\] 
Since $\gamma\geq \alpha$, we have $b_1\mathrel{\gamma}b_2$. 
Since $b_1\mathrel{\beta}b_2$ and $\gamma\wedge \beta=\alpha$,
we have that $b_1\mathrel{\alpha}b_2$. 
Thus, for each $x/\gamma\in (o/\gamma)/(\delta/\gamma)$ there 
exists a unique $b/\alpha\in (o/\alpha)/(\beta/\alpha)$ such that
$(x/\gamma, b/\alpha)\in h_o$. 

Next, we show that $h_o$ is injective. 
To this end, let $x/\gamma, y/\gamma\in (o/\gamma)/(\delta/\gamma)$ 
with $h_o(x/\gamma)=h_o(y/\gamma)$ and let $b\in A$ be such that
$h_o(x/\gamma)=b/\alpha$. 
Lemma~\ref{lemma:conseguenze_permutabilita_intervalli_proiettivi_esistenza,d(b,o,c)}
yields that there exist $c_x,c_y\in o/\gamma$ such that
$x\mathrel{\alpha}d(b,o,c_x)\mathrel{\gamma} d(b, o,c_y)\mathrel{\alpha} y$. 
Thus, $(x,y)\in \gamma\vee \alpha=\gamma$. 

Finally, $h_o$ is surjective since for each $b\in o/\beta$ we have that
$b=d(b,o,o)$ and $b\in o/\delta$, hence $h_o(b/\gamma)=b/\alpha$. 
\end{proof}
\begin{lemma}\label{lemma:isomorphisms_polynomial_restrictions_projective_intervals}
Let $\ab{A}$ be a Mal'cev algebra, let $n\in\N$,
let $o\in A$, let $\alpha,\beta,\gamma,\delta\in \Con\ab{A}$
with $\alpha\leq\beta$, $\gamma\leq\delta$ and 
$\interval{\alpha}{\beta}\nearrow\interval{\gamma}{\delta}$.
For each $p\in\POL\ari{n}\ab{A}$ with $p(o,\dots, o)=o$ we define 
\[
\begin{split}
f\colon ((o/\alpha)/(\beta/\alpha))^n\to (o/\alpha)/(\beta/\alpha) &\text{ by } 
(x_1/\alpha, \dots, x_n/\alpha)\mapsto p(\vec{x})/\alpha\\
g\colon ((o/\gamma)/(\delta/\gamma))^n\to (o/\gamma)/(\delta/\gamma) &\text{ by } 
(x_1/\gamma, \dots, x_n/\gamma)\mapsto p(\vec{x})/\gamma.
\end{split}
\]
Then the function $h_o$ from
Lemma~\ref{lemma:ismorphisms_between_classes_same_element_projective_intervals}
is an isomorphism between the algebras $((o/\alpha)/(\beta/\alpha); f)$ 
and $((o/\gamma)/(\delta/\gamma); g)$. 
\end{lemma}
\begin{proof}
Let $x_1/\gamma, \dots, x_n/\gamma\in (o/\gamma)/(\delta/\gamma)$. We show that 
$h_o(g(x_1/\gamma, \dots, x_n/\gamma))=f(h_o(x_1/\gamma), \dots, h_o(x_n/\gamma))$.
By Lemma~\ref{lemma:conseguenze_permutabilita_intervalli_proiettivi_esistenza,d(b,o,c)}
for each $i\in\finset{n}$ there exists $b_i\in o/\beta$ and 
$c_i\in o/\gamma$ such that $x_i\mathrel{\alpha} d(b_i, o,c_i)$,
and by definition $h_o(x_i/\gamma)=b_i/\alpha$. 
Since $[\beta, \gamma]\leq\alpha\leq\gamma$ Lemma~\ref{lemma:ex_citation_to_prop2.6Aic06} yields
\[
\begin{split}
&h_o(g(x_1/\gamma, \dots, x_n/\gamma))=h_o(p(\vec{x})/\gamma)=\\
&h_o(p(d(b_1,o,c_1), \dots, d(b_n, o,c_n))/\gamma)=\\
&h_o(d(p(b_1, \dots, b_n),p(o, \dots, o), p(c_1, \dots, c_n))/\gamma)=\\
&h_o(d(p(b_1, \dots, b_n),p(o,\dots, o),p(o, \dots, o))/\gamma)=\\
&h_o(p(b_1, \dots, b_n)/\gamma)=p(b_1, \dots, b_n)/\alpha=\\
&f(b_1/\alpha, \dots, b_n/\alpha)=f(h_o(x_1/\gamma), \dots, h_o(x_n/\gamma)).
\end{split}
\]
\end{proof}

\section{Abelian intervals that are simple complemented modular lattices}\label{sec:coord_simple-comple_mod}
In this section we investigate the properties of 
the algebra $\ab{A}\restrict{o/\mu}$ for $o\in A$
and $\mu\in \Con\ab{A}$ when 
$\ab{A}$ is a finite Mal'cev algebra and $\mu$ is an abelian
congruence of $\ab{A}$ such that
$\interval{\bottom{A}}{\mu}$ is a simple 
complemented modular lattice. 
For the theory of simple complemented modular lattices 
(or projective geometries) we refer the reader to
\cite[Section~4.8]{McKMcnTay88}.
First, we report a well established fact in lattice theory:
\begin{lemma}\label{lemma:Lemma_4.1_revised}
Let $\ab{L}$ be a complemented modular lattice with largest element $1$
and smallest element $0$, 
let $l\in \N$, and let $\alpha_1, \dots, \alpha_l$  be a 
set of atoms minimal with the property that 
$\alpha_1\vee\dots\vee\alpha_l=1$. 
Then the height of $\ab{L}$ is $l$ and for each $i\leq l$, 
setting $\beta_i:=\bigvee_{\substack{j=1\\j\neq i}}^{l}\alpha_j$,
we have that
$\interval{0}{\alpha_i}\nearrow \interval{\beta_i}{1}$.
\end{lemma} 
The next lemma follows from Lemma~\ref{lemma:Lemma_4.1_revised} 
setting $\ab{L}$ to be the interval $\interval{\bottom{A}}{\mu}$.
\begin{lemma}\label{lemma:set_of_generators_of_congruence_classes_abelian_congruence}
Let $\ab{A}$ be a finite Mal'cev algebra, let $\mu$ be an abelian
element of $\Con{\ab{A}}$ such that $\interval{\bottom{A}}{\mu}$
is a
simple complemented modular lattice, let $o\in A$,
let $l\in\N$, and let 
$\alpha_1, \dots, \alpha_l$ be atoms
of $\interval{\bottom{A}}{\mu}$ such that
$\bigvee_{i=1}^l\alpha_i=\mu$ and for all
$i\in \finset{l}$ setting $\beta_i:=\bigvee_{\substack{j=1\\j\neq i}}^{l}\alpha_j$ we 
have that $\beta_i<\mu$. Then for each $x\in o/\mu$ there exist
$\{b_1, \dots, b_l\}\subseteq o/\mu$ such that
for each $i\in\finset{l}$ we have $b_i\in o/\alpha_i$ and
$x=b_1+_o\dots +_o b_l$. 
\end{lemma}
\begin{proof}
We prove  the following statement by induction on $k$:
\begin{claim}\label{claim:induction_claim_lemma:set_of_generators_of_congruence_classes_abelian_congruence}
Let $k\in\finset{l}$,
and let $\gamma:=\bigvee_{i=1}^{k}\alpha_i$.
Then for all $x\in o/\gamma$ there exist $b_1, \dots, b_k\in A$
such that $x=b_1 +_o\dots +_o b_k$ and for all $j\in\finset{k}$
we have $b_j\in o/\alpha_j$. 
\end{claim}
The base step is trivial: If $k=1$, then 
$\gamma=\alpha_k$, and the claim
follows setting $b_1=x$.  
For the induction step, we assume 
Claim~\ref{claim:induction_claim_lemma:set_of_generators_of_congruence_classes_abelian_congruence}
to be true for $k-1$ and prove it for $k$.
Let $\delta=\bigvee_{i=1}^{k-1} \alpha_{i}$.
Then
Lemma~\ref{lemma:Lemma_4.1_revised}
yields
$\interval{\bottom{A}}{\alpha_{k}}\nearrow
\interval{\delta}{\gamma}$.
Thus, 
Lemma~\ref{lemma:conseguenze_permutabilita_intervalli_proiettivi_esistenza,d(b,o,c)}
yields that there exist $b_k\in o/\alpha_{k}$ and 
$e\in o/\delta$ such that
$x=e+_o b_k$. Moreover, by the induction hypothesis, 
there exist $b_1\in o/\alpha_{1}, \dots, 
b_{k-1}\in o/\alpha_{k-1}$ such that
$e=b_1+_o\dots +_o b_{k-1}$. 
Thus, Lemma~\ref{lemma:piu_emeno_fanno_gruppoo_abeliano_nella_classe_congruenza}
implies that
\[x=e+_o b_k= b_1+_o\dots +_o b_{k-1} +_o b_k.\] 
This concludes the proof of Claim~\ref{claim:induction_claim_lemma:set_of_generators_of_congruence_classes_abelian_congruence}.

Then the statement of the lemma follows from 
Claim~\ref{claim:induction_claim_lemma:set_of_generators_of_congruence_classes_abelian_congruence} 
by setting $k=l$. 
\end{proof}
The following is a partial generalization of \cite[Proposition~8.1]{AicIdz04}.
\begin{proposition}\label{prop:identification_with_matrices}
Let $\ab{A}$ be a finite Mal'cev algebra, let $\mu$ be an abelian
element of $\Con{\ab{A}}$ such that $\interval{\bottom{A}}{\mu}$ 
is a
simple complemented modular lattice of height $m$, let $o\in A$
with $\card{o/\mu}>1$,
let 
\[R_o:=\{p\restrict{o/\mu}\mid p\in \POL\ari{1}\ab{A}\text{ and } p(o)=o\}\] 
and let $\ab{R}_o$ be defined as in
\eqref{eq:equazione_che_definisce_le_operazioni_dell'anello_dei_polinomi_ristretti}
(with $\mu$ in place of $\alpha$). 
Then there exist $n\in\N$,
a field $\ab{D}$, a ring isomorphism 
$\epsilon_{\ab{R}_o}\colon \ab{R}_o\to \ab{M}_n(\ab{D})$,
and
a group isomorphism $\epsilon_\mu^o\colon (o/\mu; +_o)\to\ab{D}^{(n\times m)}$ such
that for all $r\in R_o$ and $v\in o/\mu$ we have 
$\epsilon_\mu^o (r(v))=\epsilon_{\ab{R}_o}(r)\cdot \epsilon_\mu^o(v).$
\end{proposition}
\begin{proof}
From the fact that $\interval{\bottom{A}}{\mu}$ is 
simple we infer that all prime quotients are in the same projectivity class. 
Let $\alpha$ be an atom of $\interval{\bottom{A}}{\mu}$. 
Since $\mu$ is abelian, we have that 
$[\alpha, \alpha]\leq [\mu, \mu]=\bottom{A}$.
From
Lemma~\ref{lemma:piu_meno_fanno_modulo_su-anello_polinomi_ristretto_nella_classe_congruenza}
we infer that $(o/\alpha;+_o)$ is an $\ab{R}_o$-submodule of the $\ab{R}_o$-module
$(o/\mu; +_o)$.
Moreover, since $\interval{\bottom{A}}{\mu}$ is 
a complemented modular lattice, 
\cite[Lemma~4.83]{McKMcnTay88}
implies that $\mu$ is the join of the atoms of 
$\interval{\bottom{A}}{\mu}$.  
Since every pair of prime quotients is
in the same projectivity class,
for each pair of atoms $\alpha_1, \alpha_2\in \interval{\bottom{A}}{\mu}$
we have that $\interval{\bottom{A}}{\alpha_1}\leftrightsquigarrow
\interval{\bottom{A}}{\alpha_2}$.
Thus,
Lemma~\ref{lemma:isomorphisms_polynomial_restrictions_projective_intervals}
implies that $(o/\alpha_1; +_o)$ and $(o/\alpha_2; +_o)$ are
$\ab{R}_o$-isomorphic as $\ab{R}_o$-submodules of $(o/\mu; +_o)$.  

Next, we prove that $\ab{R}_o$ acts faithfully on $(o/\alpha; +_o)$.
To this end, let $r\in R_o$ with $r(o/\alpha)=\{o\}$.
We show that $r(o/\mu)=\{o\}$. Let $x\in o/\mu$, and
let $\alpha_1, \dots, \alpha_m$ be atoms 
of $\interval{\bottom{A}}{\mu}$ such that 
$\bigvee_{i=1}^{m}\alpha_i=\mu$
and for all $i\leq m$ setting
\[\beta_i:=\bigvee_{\substack{j=1\\j\neq i}}^m \alpha_j\] we have
$\beta_i<\mu$.
The fact that any set with this property has cardinality 
$m$ follows
from 
Lemma~\ref{lemma:Lemma_4.1_revised}.
Let $b_1, \dots, b_m$ be the elements of $o/\mu$
constructed in
Lemma~\ref{lemma:set_of_generators_of_congruence_classes_abelian_congruence}.
Since all the $\ab{R}_o$-submodules $(o/\alpha_i; +_o)$
are $\ab{R}_o$-isomorphic,
we have that $r(b_i)=o$ for all $i\leq m$.
Thus, $r(x)=r(b_1+_o\dots +_o b_m)=r(b_1)+_o\dots +_or(b_m)=o$. 

Since $\alpha$ is an atom, it is generated by any 
pair of the form $(a, o)$ with $a\in o/\alpha\setminus\{o\}$.
Thus, \cite[Theorem~4.70(iii)]{McKMcnTay88} implies that 
for each pair $a,b\in o/\alpha\setminus\{o\}$ there 
exists $p\in \POL\ari{1}\ab{A}$ such that
$p(a)=b$ and $p(o)=o$. Hence, $r:=p\restrict{o/\mu}\in R_o$ 
and satisfies $r\cdot a=b$. Therefore, $(o/\alpha; +_o)$ is an irreducible 
$\ab{R}_o$-module.
Thus, since $\ab{R}_o$ acts faithfully on $(o/\alpha;+_o)$, 
$\ab{R}_o$ is primitive and therefore, by Jacobson's Density Theorem
(cf.~\cite[page~199]{BSAII}),
$\ab{R}_o$ is isomorphic to the matrix ring $\ab{M}_n(\ab{D})$
where $\ab{D}$ is the field of $\ab{R}_o$-endomorphisms of $o/\alpha$
and $n$ is the dimension of $(o/\alpha;+_o)$ over $\ab{D}$.

Since $(o/\mu;+_o)$ is the sum of $m$ simple 
$\ab{R}_o$-modules that are $\ab{R}_o$-isomorphic to 
$(o/\alpha;+_o)$, the $\ab{R}_o$-module
$(o/\mu;+_o)$ is $\ab{R}_o$-isomorphic to $(o/\alpha;+_o)^m$.
Since the module $(o/\alpha;+_o)$ is isomorphic to 
$\ab{D}^n$, we obtain that 
$(o/\mu;+_o)$ is isomorphic to $\ab{D}^{(n\times m)}$. 
\end{proof}
In the following Corollary we specialize the results from 
Proposition~\ref{prop:identification_with_matrices}
to the case when $(\ab{A},\mu)$ satisfies (AB$p$). 
\begin{corollary}\label{cor:introducing_ABp_field_abelian_congruences}
Let $\ab{A}$ be a finite Mal'cev algebra, let $\mu$ be an abelian
element of $\Con{\ab{A}}$ such that $\interval{\bottom{A}}{\mu}$ is a
finite simple complemented modular lattice of height $m$, and 
$(\ab{A}, \mu)$ satisfies (AB$p$),
let $o\in A$ with $\card{o/\mu}>1$, and
let $\ab{R}_o$ be defined as in 
\eqref{eq:equazione_che_definisce_l'universo_dell'anello_dei_polinomi_ristretti}
and \eqref{eq:equazione_che_definisce_le_operazioni_dell'anello_dei_polinomi_ristretti}.
Then there exist
a ring isomorphism 
$\epsilon_{\ab{R}_o}\colon \ab{R}_o\to \GF{p}$,
and
a group isomorphism $\epsilon_\mu^o\colon (o/\mu; +_o)\to \Z_p^m$ such
that for all $r\in R_o$ and $v\in o/\mu$ we have 
$\epsilon_\mu^o (r(v))=\epsilon_{\ab{R}_o}(r)\cdot \epsilon_\mu^o(v).$
\end{corollary}
In the proof of Proposition~\ref{prop:char_congruence_preserving_functions_module}
we will use the following observation:
\begin{lemma}\label{lemma:compoition_of_morphisms_as_suggested_by_referee}
Let $\ab{B}$ and $\ab{C}$ be two algebras
and let $\phi$ be an isomorphism from $\ab{B}$ to $\ab{C}$. 
Let $T\subseteq B$ and let $f\colon T\to B$. 
Then $f$ preserves the congruences of $\ab{B}$ if and only if 
the function $\phi\circ f\circ \phi^{-1}\colon \phi(T)\to C$ 
preserves the congruences of $\ab{C}$. 
\end{lemma}
\begin{proposition}\label{prop:char_congruence_preserving_functions_module}
Let $\ab{A}$ be a finite Mal'cev algebra,
let $\mu$ be an abelian congruence of $\ab{A}$ such 
that $\interval{\bottom{A}}{\mu}$ is a simple complemented
modular lattice of height $h$, let $o\in A$ with $\card{o/\mu}>1$,
let $T\subseteq o/\mu$, let $n$ and $\epsilon_\mu^o$ 
be the natural number and the group homomorphism constructed in 
Proposition~\ref{prop:identification_with_matrices}, 
and let $f\colon T\to 
o/\mu$. Then $f$ is congruence preserving if and only if 
$\epsilon_\mu^o\circ f\circ (\epsilon_\mu^o)^{-1}$ is a
partial congruence preserving function of the 
$\ab{M}_n(\ab{D})$-module $\ab{D}^{(n\times h)}$.
\end{proposition}
\begin{proof}
Lemma~\ref{lemma:piu_meno_fanno_modulo_su-anello_polinomi_ristretto_nella_classe_congruenza}
implies that the induced algebra 
$\ab{A}\restrict{o/\mu}$ and the 
$\ab{R}_o$-module $(o/\mu; +_o)$ 
are polynomially equivalent.
Hence these two algebras have the same congruences 
and the same partial congruence-preserving functions. 
Furthermore, the ring isomorphism
$\epsilon_{\ab{R}_o}$ between $\ab{R}_o$ and
$\ab{M}_n(\ab{D})$ constructed in 
Proposition~\ref{prop:identification_with_matrices}
defines on $\ab{D}^{(n\times m)}$ the structure
of an $\ab{R}_o$-module, and 
Proposition~\ref{prop:identification_with_matrices}
implies that this module is isomorphic to 
$(o/\mu;+_o)$ via $\epsilon_\mu^o$. 
Then the statement of the proposition follows from 
Lemma~\ref{lemma:compoition_of_morphisms_as_suggested_by_referee}.
\end{proof}
\begin{corollary}\label{cor.caratterizzazione_funzioni_compatibili_in_una_medesima_classe}
Let $\ab{A}$ be a finite Mal'cev algebra,
let $\mu$ be an abelian congruence
of $\ab{A}$ 
such that $\interval{\bottom{A}}{\mu}$ is 
a simple complemented modular lattice of height $h$
and $(\ab{A}, \mu)$ satisfies (AB2),
let $o\in A$ with $\card{o/\mu}>1$, and let $\epsilon_\mu^o$ be 
the group isomorphism built in 
Corollary~\ref{cor:introducing_ABp_field_abelian_congruences}.
Then a partial function $f\colon T\subseteq o/\mu\to o/\mu$
preserves the congruences of $\ab{A}$ 
if and only if the mapping
$\epsilon_\mu^o\circ f\circ(\epsilon_\mu^o)^{-1}\colon \epsilon_\mu^o(T)
\to \Z_2^h$
is a congruence-preserving map of the
$\GF{2}$-vector space $\Z_2^h$. 
\end{corollary}

\section{Minimal sets of abelian intervals that are simple complemented modular lattices}\label{sec:Idziak's_contributions}
In this section we investigate the properties of the minimal sets
(cf.~\cite[Definition~2.5]{HobbMcK})
of those tame quotients $\quotienttct{\bottom{A}}{\mu}$
where $\mu$ is abelian and $\interval{\bottom{A}}{\mu}$ 
is a simple complemented modular lattice. 
\begin{lemma}\label{lemma:shifting_zero_minimal_set}
Let $ \ab{A}$ be a finite Mal'cev algebra, 
let $\mu$ be an abelian congruence of
$\ab{A}$, and let $o\in A$ be such that 
there exists a $\quotienttct{\bottom{A}}{\mu}$-minimal set
$U$ contained in $o/\mu$. 
Then there exists a $\quotienttct{\bottom{A}}{\mu}$-minimal set
$U'$
such that $o\in U'$ and $U'\subseteq o/\mu$.  
\end{lemma}
\begin{proof}
If $o\in U$, we let $U':=U$.
Let us now assume that $o\notin U$. 
Then \cite[Theorem~2.8(2)]{HobbMcK} implies 
that there exists an idempotent unary polynomial 
$e$ such that $e(A)=U$. 
We let 
\[
U':=\{d(u, e(o),o)\mid u\in U\}.
\]
Since $e(o)\in U$, we have that $o=d(e(o), e(o), o)\in U'$.
Moreover, since $U\subseteq o/\mu$, for each 
$u\in U$ we have that $d(u, e(o), o)\mathrel{\mu} o$ and 
therefore, $U'\subseteq o/\mu$. 
Next, we show that the map 
\[
\{(u, d(u, e(o), o))\mid u\in U\}
\]
is injective. To this end, let $u_1, u_2\in U$.
Then $d(u_1, e(o), o)=u_1-_o e(o)$ and $d(u_2, e(o), o)
=u_2-_o e(o)$.
Thus, if $d(u_1, e(o), o)=d(u_2, e(o), o)$, then 
$u_1-_o e(o)=u_2-_o e(o)$,
and by
Lemma~\ref{lemma:piu_emeno_fanno_gruppoo_abeliano_nella_classe_congruenza}
$u_1=u_2$. 
Thus, by \cite[Theorem~2.8(3)]{HobbMcK} $U'$ is a 
$\quotienttct{\bottom{A}}{\mu}$-minimal set.
\end{proof}
The following theorem was first proved for finite 
expanded groups in 1999 during the preparation of~\cite{AicIdz04}.
That proof was never published.
In the following we give a generalization to 
finite Mal'cev algebras.
\begin{theorem}\label{teor:extending_target_minimal-set_to_congruence}
Let $ \ab{A}$ be a finite Mal'cev algebra,
let $\mu$ be an abelian congruence of
$\ab{A}$ such that $\interval{\bottom{A}}{\mu}$ is 
a finite simple complemented modular lattice, and let $o\in A$ 
be such that 
$\card{o/\mu}>1$ and there exists a 
$\quotienttct{\bottom{A}}{\mu}$-minimal set 
$V\subseteq o/\mu$. 
Then, there exists an idempotent
unary polynomial $e$ such that $e(A)=o/\mu$. 
\end{theorem}
\begin{proof}
Let $\ab{R}_o$, $\ab{D}$, $\epsilon_\mu^o$, $\epsilon_{\ab{R}_o}$,
$m$ and $n$ be as constructed in
Proposition~\ref{prop:identification_with_matrices}.
Let $\mathcal{L}_1$ be the set of all the 
$n\times m$ matrices that have non-zero entries 
only in the first row. 
For each pair $(i,j)\in\finset{n}^2$, let $E^{(i,j)}$ be the
$n\times n$ matrix that has the entry $1$ at the $(i,j)$-th place, and 
is zero elsewhere. 
Next we show that 
\begin{equation}\label{eq:eq625_from_idziak_original_manusctipt}
\text{there exists a }\quotienttct{\bottom{A}}{\mu}\text{-minimal 
set } U_1\subseteq o/\mu \text{ such that }
\epsilon_\mu^o(U_1)\subseteq \mathcal{L}_1. 
\end{equation}
Lemma~\ref{lemma:shifting_zero_minimal_set} yields a 
$\quotienttct{\bottom{A}}{\mu}$-minimal set $U$ that contains $o$
and is contained in $o/\mu$.  
Since the interval $\interval{\bottom{A}}{\mu}$ is tight
(cf.~\cite[Exercises~1.14(5)]{HobbMcK}), 
\cite[Theorem~2.11]{HobbMcK} implies that 
$\quotienttct{\bottom{A}}{\mu}$ is tame, 
and therefore, 
\cite[Theorem~2.8(2)]{HobbMcK} yields that there exists 
an idempotent unary polynomial $e_U$ with range $U$.
Let $u\in U\setminus\{o\}$, and let $M_u:=\epsilon_\mu^o(u)$.
Then $M_u$ has a non-zero entry. Let $(i,j)\in \finset{n}^2$
be such that $M_u(i,j)\neq 0$. 
Next, let $f\in \POL\ari{1}\ab{A}$ be such that 
$f\restrict{o/\mu}=(\epsilon_{\ab{R}_o})^{-1}(E^{(1,i)})$.
We have that $f(e_U(u))\neq o$, hence \cite[Theorem~2.8(3)]{HobbMcK}
yields that $U_1:=f(U)$ is a $\quotienttct{\bottom{A}}{\mu}$-minimal 
set, and since $f(o)=o$, we have $f(U)=U_1\subseteq o/\mu$.
Next, we show that 
\begin{equation}\label{eq:eq626_from_idziak_original_manusctipt}
\epsilon_\mu^o(U_1)\subseteq \mathcal{L}_1.
\end{equation}
Let $e_{U_1}$ be the idempotent polynomial with
target $U_1$ given by 
\cite[Theorem~2.8(2) and Theorem~2.11]{HobbMcK}.
We observe that the matrix $E^{(1,i)}$ has 
non-zero entries only in the first row,
a property that carries over to every matrix of the form
$E^{(1,i)}\cdot M$ for 
$M\in\ab{D}^{(n\times m)}$. 
This concludes the proof of
\eqref{eq:eq626_from_idziak_original_manusctipt}.
Hence, 
\eqref{eq:eq625_from_idziak_original_manusctipt} 
follows. 
Next, we show that 
\begin{equation}\label{eq:eq627_from_idziak_original_manusctipt}
\text{there is an idempotent } e_1\in\POL\ari{1}\ab{A} \text{ with }e_1(A)=U_1 
\text{ and } \epsilon_{\ab{R}_o}(e_1)=E^{(1,1)}. 
\end{equation}
Let $p\in\POL\ari{1}\ab{A}$ with $p(o)=o$ and
$\epsilon_{\ab{R}_o}(p)=E^{(1,1)}$. 
Then we set $e_1:=e_{U_1}\circ p$. We claim that 
$e_1(A)=U_1$ and $e_1\restrict{U_1}$ is the identity.
We have that $e_1(A)=e_{U_1}(p(A))\subseteq e_{U_1}(A)=U_1$.
Furthermore, for each $u_1\in U_1$, we have $p(u_1)=u_1$ by the 
special shape of $E^{(1,1)}$. Hence $e_1(u_1)=u_1$. 
Next, we show that $\epsilon_{\ab{R}_o}(e_1)=E^{(1,1)}$. 
Let $E_{U_1}$ be the matrix $\epsilon_{\ab{R}_o}(e_{U_1})$.
By~\eqref{eq:eq626_from_idziak_original_manusctipt} 
the range of $e_{U_1}$ is contained in 
$(\epsilon_\mu^o)^{-1}(\mathcal{L}_1)$; thus
the matrix $E_{U_1}$ has non-zero entries only in the first row.
Moreover, the entry in the first column is $1$, as $E_{U_1}$
is a non-zero idempotent matrix.
Thus, we have that
$E_{U_1}\cdot E^{(1,1)}=E^{(1,1)}$. This proves
\eqref{eq:eq627_from_idziak_original_manusctipt}.

Using the same procedure as above, for each $i\in\finset{n}$
we can construct an idempotent polynomial function
$e_i$ such that $\epsilon_{\ab{R}_o}(e_i)=E^{(i,i)}$ 
and $e_i(A)\subseteq o/\mu$. 
Next, we let 
\[e:=(\epsilon_{\ab{R}_o})^{-1}\Biggl(\sum_{i=1}^{n}\epsilon_{\ab{R}_o}(e_i)\Biggr).\] 
Clearly $e(A)\subseteq o/\mu$. 
Furthermore, $\epsilon_{\ab{R}_o}(e)$ is the identity matrix. Thus,
$e\restrict{o/\mu}$ is the identity. 
\end{proof}

\section{Homogeneous congruences in Mal'cev algebras with (SC1)} \label{sec:homegeneous_congruences}
In this section we introduce the notion of 
homogeneous congruence (cf.~\cite{IdzSlo01, AicIdz04, AicMud09}), 
and we describe the properties of homogeneous congruences in Mal'cev 
algebras with (SC1). 
Following \cite{AicMud09} we define:
\begin{definition}[{\cite[Definition~8.1]{AicMud09}}]\label{def:APMI}
A Mal'cev algebra $\ab{A}$ satisfies the condition (APMI) if 
for all strictly meet irreducible congruences $\alpha, \beta$ of 
$\ab{A}$ such that 
$\interval{\alpha}{\alpha^+}\leftrightsquigarrow\interval{\beta}{ \beta^+} $, we have  
$\alpha^+=\beta^+$. 
\end{definition}
The following proposition was first 
established for expanded groups in \cite{AicMud09}.
However the proof generalizes almost verbatim to 
Mal'cev algebras. 
\begin{proposition}\label{prop:SC1impliesAPMI}
Let $\ab{A}$ be a Mal'cev algebra. If $\ab{A}$ satisfies (SC1), 
then $\ab{A}$ satisfies (APMI).
\end{proposition}
\begin{proof}
Let $\alpha, \beta$ be strictly meet irreducible in $\Con\ab{A}$ with
$\interval{\alpha}{\alpha^+}\leftrightsquigarrow \interval{\beta} {\beta^+}$.
Then, Lemma~\ref{lemma:exAic18Lemma3.4} yields that 
$(\alpha: \alpha^+)=(\beta:\beta^+)$, and (SC1) implies 
that $(\alpha: \alpha^+)\in\{\alpha,\alpha^+\}$. 
If $(\alpha:\alpha^+)=\alpha^+$, then 
$[\alpha^+, \alpha^+]\leq \alpha$, 
and therefore, Lemma~\ref{lemma:exAic18Lemma3.4} implies 
$[\beta^+, \beta^+]\leq \beta$. Thus $(\beta: \beta^+)\geq \beta^+$,
and thus, (SC1) implies that $(\beta:\beta^+)=\beta^+$. 
Hence $\beta^+=\alpha^+$. 
On the other hand, if $(\alpha: \alpha^+)=\alpha$, 
then $[\alpha^+, \alpha^+]\nleq \alpha$.
Thus, Lemma~\ref{lemma:exAic18Lemma3.4} implies that 
$[\beta^+, \beta^+]\nleq \beta$. Hence $(\beta:\beta^+)=\beta$,
and therefore, $\alpha=\beta$, 
which clearly implies that $\beta^+=\alpha^+$.  
\end{proof}
\begin{definition}[{\cite[Definition~7.2]{AicMud09}}]\label{def:homogeneous}
Let $\ab{L}$ be a bounded lattice with smallest element $0$.
An element $\mu$ of $L$ is 
\emph{homogeneous} if 
\begin{enumerate}
\item $\mu>0$,
\item for all $\alpha, \beta, \gamma, \delta\in L$ 
with $\alpha\prec \beta \leq \mu$ and $\gamma\prec \delta \leq \mu$, 
we have 
$\interval{\alpha}{\beta}\leftrightsquigarrow\interval{\gamma}{\delta}$, and
\item there are no $\alpha, \beta, \gamma,\delta\in L$ such that 
$\alpha\prec\beta\leq\mu\leq\gamma\prec\delta$, and
$\interval{\alpha}{\beta}\leftrightsquigarrow\interval{\gamma}{\delta}$. 
\end{enumerate}
\end{definition}
Following \cite{AicMud09}, for an element $\mu$ of a complete lattice $\ab{L}$
with smallest element $0$, we define 
\[
\begin{split}
\Phi(\mu)&:=\mu\wedge \bigwedge\{\alpha\in L\mid \alpha\prec \mu\};\\
\mu^*&:=\bigvee\{\alpha\in L\mid \alpha\wedge \mu=0\}.
\end{split}
\]
Note that if $\ab{L}$ has finite height
and $\mu$ is not $0$, then 
\[\Phi(\mu)=\bigwedge\{\alpha\in L\mid \alpha\prec \mu\}.\] 
\begin{lemma}\label{lemma:whymumeetmustariszero}
Let $\ab{L}$ be a modular lattice of finite height with
smallest element $0$, and let
$\mu$ be a homogeneous element of $\ab{L}$. Then $\mu\wedge\mu^*=0$.  
\end{lemma}
\begin{proof}
Since $\ab{L}$ is of finite height there is a finite
subset $\{\alpha_1,\dots, \alpha_n\}$ of $L$ such that
$\mu^*=\alpha_1\vee\dots \vee \alpha_n$ and 
$\alpha_i\wedge\mu=0$
for all $i\in\finset{n}$. Since $\ab{L}$ is algebraic,
\cite[Lemma~7.3]{AicMud09} yields that $\mu$ is 
a distributive element of $\ab{L}$, 
and therefore, dually distributive by 
\cite[p.~187, Theorem~6]{Gra98}.
Thus, $\mu\wedge \mu^*=0$. 
\end{proof}
\begin{proposition}\label{prop:homogeneueselementsplittingwithcentralizer}
Let $\ab{A}$ be a Mal'cev algebra whose congruence lattice has 
finite height, and let $\mu$ be a homogeneous congruence
of $\ab{A}$. Then for all $\alpha\in \Con\ab{A}$ we have: 
$\alpha\geq \mu$ or $\alpha\leq (\Phi(\mu):\mu)$.
\end{proposition}
\begin{proof}
Let $\alpha\in\Con\ab{A}$ with $\alpha\ngeq \mu$. We show that $\alpha\leq (\Phi(\mu):\mu)$. 
We have that $[\mu,\alpha]< \mu$. Thus, there exists $\eta\in \Con\ab{A}$ such
that $[\mu, \alpha]\leq \eta\prec\mu$. Thus, $\alpha\leq(\eta:\mu)$. 
Next, we show that $(\eta:\mu)=(\Phi(\mu):\mu)$. 
Since $\eta\geq \Phi(\mu)$, we have $(\eta:\mu)\geq (\Phi(\mu):\mu)$.
Next we show  the opposite inequality. To this end,
we first prove that each subcover $\eta'$ of $\mu$ satisfies 
$[(\eta: \mu), \mu]\leq \eta'$.
Since $\mu$ is homogeneous, we have that
$\interval{\eta}{\mu}\leftrightsquigarrow
\interval{\eta'}{\mu }$; then Lemma~\ref{lemma:exAic18Lemma3.4} yields 
that $(\eta:\mu)=(\eta':\mu)$. This concludes the proof that 
$[(\eta:\mu), \mu]$ is less than or equal to each subcover of $\mu$. 
Thus, $[(\eta:\mu),\mu]\leq\Phi(\mu)$, and therefore,
$(\eta:\mu)\leq(\Phi(\mu):\mu)$. 
\end{proof}
The results from \cite[Section~3]{Aic18} allow us to extend 
\cite[Proposition~9.6]{AicMud09}
and \cite[Proposition~10.1(2)]{AicMud09}
from finite expanded groups
to finite Mal'cev algebras. 
\begin{proposition}\label{prop:typesandhomogenuity}
Let $\ab{A}$ be a Mal'cev algebra whose congruence 
lattice has finite height, let $\mu$ be a homogeneous 
element of $ \Con\ab{A}$, and let $\alpha,\beta\in \Con\ab{A}$ 
with $\alpha\prec\beta\leq\mu$. Then we have:
\begin{enumerate}
\item If $[\beta,\beta]\nleq\alpha$, then 
$\alpha=\bottom{A}$, $\beta=\mu$, and $(\bottom{A}:\mu)=\mu^*$. 
\label{item:typethreecasehomogeneous}
\item If $[\beta,\beta]\leq\alpha$, then
$(\alpha :\beta)=(\Phi(\mu):\mu)\geq \mu\vee\mu^*$. \label{item:typetwocasehomogeneous}
\end{enumerate}
\end{proposition}
\begin{proof}
We first prove the following:
\begin{equation}\label{eq:projectivity_implies_type_2}
\begin{split}
&\forall \kappa_1, \kappa_2, \kappa_3, \kappa_4\in \Con\ab{A}\colon \\
&\bigl(\kappa_1\prec\kappa_2\leq \kappa_3\prec\kappa_4 \,\wedge\, 
\interval{\kappa_1}{\kappa_2}\leftrightsquigarrow \interval{\kappa_3}{\kappa_4}\bigr)
\Rightarrow [\kappa_2,\kappa_2]\leq\kappa_1.
\end{split}
\end{equation}
We have that $(\kappa_3:\kappa_4)\geq \kappa_3$ and 
Lemma~\ref{lemma:exAic18Lemma3.4}  
yields that $(\kappa_1\colon \kappa_2)\geq \kappa_3$.
Hence the monotonicity of the commutator implies 
$[\kappa_2,\kappa_2]\leq [\kappa_3, \kappa_2]\leq \kappa_1$. 

In order to prove \eqref{item:typethreecasehomogeneous}
we first show that if $\alpha\neq \bottom{A}$, then $[\beta,\beta]\leq\alpha$.
If $\alpha\neq \bottom{A}$, then there exist 
$\eta_1, \eta_2\in \Con\ab{A}$
such that $\eta_1\prec\eta_2\leq\alpha\prec\beta\leq\mu$. 
Since $\mu$ is homogeneous,  \eqref{eq:projectivity_implies_type_2}
yields that $[\eta_2,\eta_2]\leq\eta_1$, and
Lemma~\ref{lemma:exAic18Lemma3.4} implies
that $[\beta,\beta]\leq\alpha$. 
Next, we show that if $\beta\neq \mu$, then $[\beta,\beta]\leq\alpha$.
If $\beta\neq \mu$, then there exist $\eta_1, \eta_2\in \Con\ab{A}$ 
such that $\alpha\prec\beta\leq \eta_1\prec\eta_2\leq\mu$. 
Hence \eqref{eq:projectivity_implies_type_2}
yields that $[\beta,\beta]\leq\alpha$. 
Since $[\beta,\beta]\nleq\alpha$ implies $\alpha=\bottom{A}$ 
and $\beta=\mu$, we infer that $[\beta,\beta]\nleq\alpha$
also implies that for all $\kappa\in\Con\ab{A}$, we have 
\[[\kappa, \mu]=\bottom{A}\Leftrightarrow\kappa\wedge \mu=\bottom{A},\] 
and therefore, $(\bottom{A}:\mu)=\mu^*$. 
This concludes the proof of \eqref{item:typethreecasehomogeneous}.

Next, we show \eqref{item:typetwocasehomogeneous}. To this end, 
let $\eta\prec\mu$. Since $\mu$ is homogeneous, 
we have that $\interval{\alpha}{\beta}\leftrightsquigarrow\interval{\eta}{\mu}$.
Thus, since $[\beta,\beta]\leq\alpha$, we have $[\mu, \mu]\leq \eta$. 
Hence 
$[\mu, \mu]\leq \Phi(\mu)$. Moreover, \cite[Lemma~4.149]{McKMcnTay88} and 
Lemma \ref{lemma:whymumeetmustariszero}
yield that $[\mu, \mu^*]\leq \mu\wedge \mu^*=\bottom{A}$.
Thus, \cite[Lemma~2.5]{Aic06} implies that $[\mu, \mu\vee\mu^*]\leq \Phi(\mu)$,
and therefore, $(\Phi(\mu):\mu)\geq \mu\vee \mu^*$. 
Next, we prove that $(\alpha:\beta)=(\Phi(\mu):\mu)$. 
We observe that since $\mu\neq \bottom{A}$ and since 
$\Con\ab{A}$ has finite height, 
$\Phi(\mu)=\bigwedge\{\eta\mid \eta\prec\mu\}$ and 
therefore,
$(\Phi(\mu):\mu)=\bigwedge\{(\eta:\mu)\mid \eta\prec\mu\}$. 
Since $\mu$ is homogeneous, Lemma~\ref{lemma:exAic18Lemma3.4}
implies that for each subcover $\eta$ of $\mu$ 
we have $(\eta:\mu)=(\alpha:\beta)$. Hence
$(\Phi(\mu):\mu)=(\alpha:\beta)$.
\end{proof}
\begin{proposition}\label{prop:homogenuityandcolpementation}
Let $\ab{A}$ be a Mal'cev algebra with (APMI)
whose congruence lattice has finite height. Then 
$\Con \ab{A}$ has a homogeneous element.
Moreover, let $\mu$ be an homogeneous element of $\Con\ab{A}$.
Then we have:
\begin{enumerate}
\item $\interval{\bottom{A}}{\mu}$ is 
a simple complemented modular lattice;\label{item:simplecomplmodlatticebelowhom}
\item $\Phi(\mu)=\bottom{A}$;\label{item:phiofhomogenous}
\item for all $\alpha\in \Con\ab{A}$ we have $\alpha\geq \mu$ or $\alpha\leq \mu\vee \mu^*$. \label{item:splithomogeous}
\end{enumerate}
\end{proposition}
\begin{proof}
First, we remark that $\Con\ab{A}$ is modular. 
Thus, \cite[Proposition~8.9]{AicMud09} yields that there exists
an homogeneous congruence. 
Items \eqref{item:simplecomplmodlatticebelowhom} 
and \eqref{item:phiofhomogenous} follow from
\cite[Proposition~8.9]{AicMud09}; 
item
\eqref{item:splithomogeous} follows from
\cite[Proposition~8.10]{AicMud09}. 
\end{proof}

\begin{proposition}\label{prop:the_centralizer_of_a_homogeneous_congruence}
Let $\ab{A}$ be a finite Mal'cev algebra with (SC1), and let 
$\mu$ be a homogeneous element of $\Con\ab{A}$. Then
$\Phi(\mu)=\bottom{A}$ and $(\Phi(\mu):\mu)\leq \mu\vee \mu^*$.  
\end{proposition}
\begin{proof}
Proposition~\ref{prop:SC1impliesAPMI} implies that $\ab{A}$ satisfies (APMI),
and Proposition~\ref{prop:homogenuityandcolpementation} yields $\Phi(\mu)=\bottom{A}$. 
Let $\alpha,\beta\in\Con{A}$ with $\alpha\prec\beta\leq\mu$. If $[\beta,\beta]\nleq\alpha$,
then Proposition~\ref{prop:typesandhomogenuity} implies that $(\bottom{A}:\mu)=\mu^*$ and the 
statement follows. 
Let us now consider the case that $[\beta,\beta]\leq\alpha$. 
From Proposition~\ref{prop:typesandhomogenuity} we infer that 
\[(\Phi(\mu):\mu)=(\alpha:\beta).\] 
Let $\eta$ be a strictly meet irreducible congruence of $\ab{A}$ with $\eta\geq \alpha$ 
and $\eta\ngeq\beta$. Then, \cite[Proposition~7.1]{AicMud09} yields that 
$\interval{\alpha}{\beta}\nearrow\interval{\eta}{\eta^+}$. 
Hence Lemma~\ref{lemma:exAic18Lemma3.4} implies that $(\alpha:\beta)=(\eta:\eta^+)$.
By (SC1) we have $(\eta:\eta^+)=\eta^+$. Since $\beta\leq \mu$ and $\eta\ngeq \beta$, 
we have $\eta\ngeq\mu$, and therefore 
Proposition~\ref{prop:homogenuityandcolpementation}  
yields that
$\eta\leq\mu\vee\mu^*$. Since $\eta^+=\beta\vee\eta$, 
we also have $\eta^+\leq \mu\vee\mu^*$. Thus, we have
$(\Phi(\mu):\mu)=\eta^+\leq\mu\vee\mu^*$.  
\end{proof}
\begin{proposition}\label{prop:if_an_homogeneues_congruence_is_not_neutral_then_it_is_abelian}
Let $\ab{A}$ be a finite Mal'cev algebra with (SC1), and let
$\mu$ be a homogeneous element of $\Con\ab{A}$ with
$[\mu,\mu]\neq \mu$. Then $(\Phi(\mu):\mu)=\mu\vee\mu^*$ and
$[\mu,\mu]=\bottom{A}$. 
\end{proposition}
\begin{proof}
Proposition~\ref{prop:the_centralizer_of_a_homogeneous_congruence}
yields that $\Phi(\mu)=\bottom{A}$ and $(\bottom{A}:\mu)\leq\mu\vee\mu^*$.
Moreover, if $[\mu,\mu]\neq \mu$, there exists $\eta\in\Con\ab{A}$
such that $[\mu,\mu]\leq\eta\prec\mu$. 
Thus Proposition~\ref{prop:typesandhomogenuity}
\eqref{item:typetwocasehomogeneous} yields that 
$(\Phi(\mu):\mu)\geq \mu\vee\mu^*$. 
Thus, we have that $(\bottom{A}:\mu)=\mu\vee\mu^*$. 
Hence we have $(\Phi(\mu):\mu)\geq \mu$, and therefore, 
$[\mu,\mu]\leq\Phi(\mu)=\bottom{A}$. 
\end{proof}
The following Lemma follows from the proof of 
\cite[Theorem~24]{IdzSlo01}, in particular equations 
(25) and (26). 
Note that the fact that $\quotienttct{\bottom{A}}{ \mu}$ is tame
follows from 
Proposition~\ref{prop:homogenuityandcolpementation}\eqref{item:simplecomplmodlatticebelowhom}.
\begin{lemma}[{\cite{IdzSlo01}}]\label{lemma:eq_24_25_26_from_IdzSlo01}
Let $ \ab{A}$ be a finite Mal'cev algebra with (SC1),
let $\mu$ be an abelian homogeneous congruence of
$\ab{A}$. Then the quotient $\quotienttct{\bottom{A}}{ \mu}$ is tame, 
and for each atom $\alpha$ of $\interval{\bottom{A}}{\mu}$ we have
$\Mtct{\ab{A}}{\bottom{A}}{\mu}=\Mtct{\ab{A}}{\bottom{A}}{\alpha}$.
Moreover, every $\quotienttct{\bottom{A}}{\mu}$-minimal set
contains only one $\quotienttct{\bottom{A}}{\mu}$-trace.
\end{lemma}

\begin{lemma}\label{lemma:minimal_set_contained_in_one_class}
Let $ \ab{A}$ be a finite Mal'cev algebra with (SC1), 
let $\mu$ be an abelian homogeneous congruence of
$\ab{A}$, let $o\in A$, and let 
$U\in \Mtct{\ab{A}}{\bottom{A}}{\mu}$ such that 
$o\in U$. Then $U\subseteq o/\mu$. 
\end{lemma}
\begin{proof}
Lemma~\ref{lemma:eq_24_25_26_from_IdzSlo01} yields 
that $U$ has only one trace, and \cite[Theorem~8.5]{HobbMcK}
implies that $U$ has empty tail. 
\end{proof}
\begin{theorem}\label{teor:existence_of_idempotent_class_homogeneous}
Let $\ab{A}$ be a finite Mal'cev algebra
with (SC1), let $\mu$ be a homogeneous abelian congruence of
$\ab{A}$, and let $v\in A$. Then there exists 
an idempotent polynomial function $e_\mu^v$ of $\ab{A}$ such that 
$e_\mu^v(A)=v/\mu$. 
\end{theorem}
\begin{proof}
If $\card{v/\mu}=1$, then the constant polynomial function
with constant value $v$ satisfies the requirements. 

Next, we assume that $\card{v/\mu}\geq 2$. Then there 
exist $u\in A$ and an atom $\alpha$ of 
$\interval{\bottom{A}}{\mu}$ such that $(v, u)\in\alpha$
and $u\neq v$. Thus, 
\cite[Exercise~8.8(1)]{HobbMcK} implies that 
there exist a 
$\quotienttct{\bottom{A}}{\alpha}$-minimal set $U$ and
a $\quotienttct{\bottom{A}}{\alpha}$-trace $N$
such that $\{u,v\}\subseteq N\subseteq U$.
Since by Lemma~\ref{lemma:eq_24_25_26_from_IdzSlo01}
the $\quotienttct{\bottom{A}}{\alpha}$-minimal sets 
are $\quotienttct{\bottom{A}}{\mu}$-minimal sets,
we infer that $U$ is a 
$\quotienttct{\bottom{A}}{\mu}$-minimal set 
that contains $v$.
Then, Lemma~\ref{lemma:minimal_set_contained_in_one_class} implies that
$U\subseteq v/\mu$. Moreover,
Proposition~\ref{prop:homogenuityandcolpementation}\eqref{item:simplecomplmodlatticebelowhom} 
implies that 
$\interval{\bottom{A}}{\mu}$ is a simple complemented 
modular lattice.
Thus, the existence of the desired 
idempotent polynomial function follows from 
Theorem~\ref{teor:extending_target_minimal-set_to_congruence}. 
\end{proof}
\begin{theorem}\label{teor:ABp_Implies_polynomial_equivalence}
Let $\ab{A}$ be a finite Mal'cev algebra with (SC1),
let $\mu\in \Con\ab{A}$ be abelian and homogeneous, and
let us assume that $(\ab{A}, \mu)$ satisfies (AB$p$).
Then each $\mu$-class that is not a singleton is 
a $\quotienttct{\bottom{A}}{\mu}$-minimal set.
\end{theorem} 
\begin{proof}
Note that 
Proposition~\ref{prop:homogenuityandcolpementation}\eqref{item:simplecomplmodlatticebelowhom} 
implies that 
$\interval{\bottom{A}}{\mu}$ is a simple complemented 
modular lattice.
Let $h$ be the height of $\interval{\bottom{A}}{\mu}$, 
and let us fix $v\in A$ with $\card{v/\mu}>1$.
Then
Corollary~\ref{cor:introducing_ABp_field_abelian_congruences}
implies that $\ab{A}\restrict{v/\mu}$ is polynomially equivalent to
a $\GF{p}$-vector space of dimension $h$. 
Moreover, since  $\card{v/\mu}>1$,
there exist
an atom $\alpha$ of $\interval{\bottom{A}}{\mu}$
and $u\in A\setminus\{v\}$ such that $(v, u)\in \alpha$.
Thus, 
\cite[Exercise~8.8(1)]{HobbMcK} implies that there exist a 
$\quotienttct{\bottom{A}}{\alpha}$-minimal set $U$ and
a $\quotienttct{\bottom{A}}{\alpha}$-trace $N$
such that $\{u,v\}\subseteq N\subseteq U$.
Hence by Lemma~\ref{lemma:eq_24_25_26_from_IdzSlo01},
$U$ is a $\quotienttct{\bottom{A}}{\mu}$-minimal set 
that contains $v$.
Then, Lemma~\ref{lemma:minimal_set_contained_in_one_class} implies that
$U\subseteq v/\mu$.
Finally, let $e$ be the idempotent polynomial constructed in 
\cite[Theorem~2.8(2)]{HobbMcK}, that satisfies $e(A)=U$.
Since $v\in U$, we have that $e\restrict{v/\mu}\in \ab{R}_v$
(defined as in \eqref{eq:equazione_che_definisce_l'universo_dell'anello_dei_polinomi_ristretti}
with $v$ in place of $o$).
Let $\epsilon_{\ab{R}_v}$ be the 
ring isomorphism constructed in
Corollary~\ref{cor:introducing_ABp_field_abelian_congruences},
and let
$\lambda:=\epsilon_{\ab{R}_v}(e\restrict{v/\mu})\in \GF{p}$. 
Since $e(u)=u\neq v$, $\lambda\neq 0$, and since $e$ is 
idempotent, $\lambda$ is idempotent, and therefore, 
$\lambda=1$. 
In view of Corollary~\ref{cor:introducing_ABp_field_abelian_congruences}
we have 
$\epsilon_\mu^v(e(x))=\lambda\epsilon_\mu^v(x)=\epsilon_\mu^v(x)$
for every $x\in v/\mu$, and since
$\epsilon_\mu^v$ is bijective, $e(x) = x$ 
for every $x\in v/\mu$.
Thus,
$U=v/\mu$, and $v/\mu$ is 
a $\quotienttct{\bottom{A}}{\mu}$-minimal set. 
\end{proof}

\section{Remarks on a well known strategy for interpolating partial functions with polynomials}\label{sec:interpolation}
A classical strategy  for interpolating congruence-preserving functions with 
polynomial functions is the following (cf.~\cite{AicIdz04, AicMud09}): 
First, one proves that the assumptions on the algebra under
consideration carry over to 
its quotients modulo a homogeneous congruence;
then one proceeds by induction and assumes that the function at hand
can be interpolated modulo a homogeneous congruence.
In the present and in the next section
we follow this pattern to interpolate congruence-preserving 
functions in Mal'cev algebras with (SC1) and (AB2). 

First, we fix some definitions. Let $\ab{A}$ be a Mal'cev algebra, 
let $k\in\N$, let $D\subseteq A^k$
let $T\subseteq D$, let $f\colon D\to A$,
and let $\mu\in \Con\ab{A}$.
We say that a polynomial function $p\in\POL\ari{k}\ab{A}$
\emph{interpolates $f$ on $T$ modulo $\mu$} if for all 
$\vec{t}\in T$ we have $p(\vec{t})\mathrel{\mu}f(\vec{t})$.
Moreover, we say that $p$ \emph{interpolates $f$ on $T$}
if $p$ interpolates $f$ on $T$ modulo $\bottom{A}$.
Finally, we say that \emph{$p$ interpolates $f$} if $p$
interpolates $f$ on $D$. 

The following lemma 
tells us how to combine  
``interpolation of functions that are constant modulo $\mu$''
and
``interpolation modulo $\mu$''
to interpolate a congruence-preserving function. 
This has been done in~\cite[Lemma~6]{KaaMay10} 
for total functions on finite sets. Here
we provide a generalization to partial functions 
on possibly infinite sets. 
\begin{lemma}\label{lemma:fromcosetstotheinfinirtyandbeyond}
Let $\ab{A}$ be a Mal'cev algebra,
let $\mu\in \Con\ab{A}$, and
let $\strset{R}$ be a set of relations on $A$ such that
$\Con\ab{A}\subseteq \strset{R}\subseteq \Inv\POL\ab{A}$.
We assume that
every partial function with finite domain
that preserves $\strset{R}$ and 
has image contained in one $\mu$-equivalence class can be interpolated by a polynomial 
function of $\ab{A}$.  
Let $k\in \N$, let $T$ be a finite subset of $A^k$, 
and let $f\colon T\to A$ be a
partial function that preserves $\strset{R}$. We assume 
that there exists $p_0\in \POL\ari{k}\ab{A}$ 
that interpolates $f$ on $T$ modulo $\mu$. 
Then there exists $p\in \POL\ari{k}\ab{A}$ that interpolates $f$ on $T$. 
\end{lemma}
\begin{proof}
We show that 
for each subset $U$ of $T$ there exists 
a polynomial $p_U\in \POL\ari{k}\ab{A}$ such that
$p_U$ interpolates $f$ on $U$ and $p_U$ interpolates 
$f$ on $T$ modulo $\mu$.
We proceed by induction on the cardinality of $U$.  

\textbf{BASE}: $\card{U}=1$: Let $\vec{t}_1$ be the unique element 
of $U$. 
We construct $p_1\in \POL\ari{k}\ab{A}$ such that 
for all $\vec{t}\in T$ we have $(p_1(\vec{t}), f(\vec{t}))\in\mu$ 
and $f(\vec{t}_1)=p_1(\vec{t}_1)$. To this end, let $\tilde{f}\colon T\to p_0(\vec{t}_1)/\mu$
be defined by $\vec{t}\mapsto d(f(\vec{t}), p_0(\vec{t}), p_0(\vec{t}_1))$. 
Then, $\tilde{f}$ preserves $\strset{R}$ as it is a composition of functions 
that preserve $\strset{R}$. Hence there exists $\tilde{p}_0\in \POL\ari{k}\ab{A}$ that 
interpolates $\tilde{f}$. 
Let us define $p_1\colon A^k\to A$ by \[\vec{a}\mapsto d(\tilde{p}_0(\vec{a}), p_0(\vec{t}_1), p_0(\vec{a})).\]
Clearly, $p_1\in\POL\ari{k}\ab{A}$. Moreover we have
	\[
	\begin{split}
		p_1(\vec{t}_1)&=d(\tilde{p}_0(\vec{t}_1),p_0(\vec{t}_1), p_0(\vec{t}_1))\\
		&=d(d(f(\vec{t}_1),p_0(\vec{t}_1),p_0(\vec{t}_1)), p_0(\vec{t}_1), p_0(\vec{t}_1))\\
		&=d(f(\vec{t}_1),  p_0(\vec{t}_1), p_0(\vec{t}_1))=f(\vec{t}_1).
	\end{split}
	\] 
Furthermore, for all $\vec{t}\in T$ we have
	\[
	\begin{split}
		p_1(\vec{t})&=d(\tilde{p}_0(\vec{t}),p_0(\vec{t}_1), p_0(\vec{t}))\\
		&=d(d(f(\vec{t}),p_0(\vec{t}),p_0(\vec{t}_1)), p_0(\vec{t}_1), p_0(\vec{t}))\\
		&\mathrel{\mu}d(p_0(\vec{t}_1),  p_0(\vec{t}_1), p_0(\vec{t}))=p_0(\vec{t})\mathrel{\mu} f(\vec{t}).
	\end{split}
	\] 
This concludes the proof of the base of the induction.

\textbf{INDUCTION STEP}: We assume that 
$U=\{\vec{t}_1,\dots, \vec{t}_i, \vec{t}_{i+1}\}$ with
$i\geq 1$ and we assume that  
there exists $p_i\in\POL\ari{k}\ab{A}$ 
that interpolates $f$ on $\{\vec{t}_1,\dots, \vec{t}_i\}$ and interpolates 
$f$ on $T$ modulo $\mu$. We show that 
there exists $p_{i+1}\in\POL\ari{k}\ab{A}$ that interpolates $f$ on 
$\{\vec{t}_1,\dots, \vec{t}_{i+1}\}$
and that interpolates $f$ on $T$ modulo $\mu$.
First, we define $\tilde{f}_i \colon T\to A$ by 
$\vec{t}\mapsto d(f(\vec{t}),p_i(\vec{t}),p_{i}(\vec{t}_{i+1}))$.
Since for all $\vec{t}\in T$ we have $(f(\vec{t}),p_i(\vec{t}))\in\mu$, by the 
induction hypothesis 
we have that the image of $\tilde{f}_i$ is a subset of $ p_i(\vec{t}_{i+1})/\mu$.
Moreover, since $f$ preserves $\strset{R}$ and since $\strset{R}\subseteq \Inv\POL\ab{A}$, 
we have that $\tilde{f}_i$ preserves $\strset{R}$. Thus, there exists 
$\tilde{p}_i\in\POL\ari{k}\ab{A}$ such that for all $\vec{t}\in T$ we have 
$\tilde{f}_i(\vec{t})=\tilde{p}_i(\vec{t})$. 
Next, let us define $p_{i+1}\colon A^k\to A$ by 
\[\vec{a}\mapsto d(\tilde{p}_i(\vec{a}), p_i(\vec{t}_{i+1}), p_i(\vec{a})).\]
We show that $p_{i+1}$ 
interpolates $f$ on 
$\{\vec{t}_1,\dots, \vec{t}_{i+1}\}$
and interpolates $f$ on $T$ modulo $\mu$.
Clearly, $p_{i+1}\in\POL\ari{k}\ab{A}$, as it is a composition of polynomial functions. Moreover, for
all $j\leq i$ we have
\[
\begin{split}
	p_{i+1}(\vec{t}_j)&=d(\tilde{p}_i(\vec{t}_j),p_i(\vec{t}_{i+1}), p_i(\vec{t}_j))\\
	&=d(d(f(\vec{t}_j),p_i(\vec{t}_j),p_i(\vec{t}_{i+1})), p_i(\vec{t}_{i+1}), p_i(\vec{t}_j))\\
	&=d(p_i(\vec{t}_{i+1}),  p_i(\vec{t}_{i+1}), p_i(\vec{t}_j))=p_i(\vec{t}_j)=f(\vec{t}_j).
\end{split}
\] 
Furthermore, for all $\vec{t}\in T$ we have
\[
\begin{split}
	p_{i+1}(\vec{t})&=d(\tilde{p}_i(\vec{t}),p_i(\vec{t}_{i+1}), p_i(\vec{t}))\\
	&=d(d(f(\vec{t}),p_i(\vec{t}),p_i(\vec{t}_{i+1})), p_i(\vec{t}_{i+1}), p_i(\vec{t}))\\
	&\mathrel{\mu}d(p_i(\vec{t}_{i+1}),  p_i(\vec{t}_{i+1}), p_i(\vec{t}))=p_i(\vec{t})\mathrel{\mu} f(\vec{t}).
\end{split}
\] 
Finally, we have
\[
\begin{split}
	p_{i+1}(\vec{t}_{i+1})&=d(\tilde{p}_i(\vec{t}_{i+1}),p_i(\vec{t}_{i+1}), p_i(\vec{t}_{i+1}))\\
	&=d(d(f(\vec{t}_{i+1}),p_i(\vec{t}_{i+1}),p_i(\vec{t}_{i+1})), p_i(\vec{t}_{i+1}), p_i(\vec{t}_{i+1}))\\
	&=d(f(\vec{t}_{i+1}),  p_i(\vec{t}_{i+1}), p_i(\vec{t}_{i+1}))=f(\vec{t}_{i+1}).
\end{split}
\] 
This concludes the proof of 
the induction step.
\end{proof}
Lemma~\ref{lemma:fromcosetstotheinfinirtyandbeyond} motivates the study of those 
partial functions whose image is contained inside one equivalence class of a congruence.
In this section we
generalize the results from \cite[Section~7]{AicIdz04} to finite Mal'cev algebras with (SC1). 
We start with a lemma that can be viewed as a partial 
generalization of \cite[Proposition~3.1]{Aic00}.

\begin{lemma}\label{lemma.come_si_fanno_le_induzioni}
Let $\ab{A}$ be an algebra with Mal'cev polynomial $d$, 
and let us assume that $\Con\ab{A}$ has finite height.
Let $\mu$ be a homogeneous element of $\Con\ab{A}$, 
let $o\in A$, let $U=o/\mu$, let $k\in\N$, let $T$ be a finite
subset of $A^k$ with at least three elements, and let us assume
that there are $\vec{t}_1,\vec{t}_2\in T$ such that 
\[
\beta:=\Cg{\{(\vec{t}_1(1),\vec{t}_2(1)),\dots, (\vec{t}_1(k),\vec{t}_2(k))\}}\nleq (\Phi(\mu):\mu).
\]
Let $l\colon A^k\to U$, and let us assume that $l$ 
can be interpolated  by a polynomial function of $\ab{A}$
whose image is contained in $U$
on each subset of $T$ with cardinality $\card{T}-1$.
Then $l$ can be interpolated on $T$ by a 
polynomial function of $\ab{A}$ whose image is contained in $U$.
\end{lemma}
\begin{proof}
Let us consider the following two sets
\begin{align*}
\eta&=\Biggl\{(p(\vec{t}_1),q(\vec{t}_1))\mathrel{\Bigg\vert} p,
q\in \POL\ari{k}\ab{A},
\begin{aligned}[c]
&\forall \vec{x}\in A^k\colon& p(\vec{x})&\mathrel{\mu}q(\vec{x}),\\
&\forall i\in\{2,\dots , n\}\colon& p(\vec{t}_i)&=q(\vec{t}_i)
\end{aligned}
\Biggr\};\\
\alpha&=\Biggl\{(p(\vec{t}_1),q(\vec{t}_1))\mathrel{\Bigg\vert} p,
q\in \POL\ari{k}\ab{A},
\begin{aligned}
&\forall \vec{x}\in A^k\colon& p(\vec{x})&\mathrel{\mu} q(\vec{x}),\\
&\forall i\in\{3,\dots, n\}\colon& p(\vec{t}_i)&=q(\vec{t}_i)
\end{aligned}
\Biggr\}.
\end{align*}
It is easy to see that~$\eta$ and~$\alpha$ are reflexive
subuniverses of $\ab{A}\times\ab{A}$.
Now, by \cite[Lemma~5.22]{HobbMcK}, we have
$\alpha,\eta\in\Con\ab{A}$, and by their definition, that
$\alpha,\eta\in\interval{\bottom{A}}{\mu}$.
Our next goal is to prove that $\alpha=\eta$.
\par
First, we prove that $\alpha$ centralizes $\beta$ modulo $\eta$. 
By \cite[Proposition~2.3]{Aic06} we only need to prove that
for all $(u,v)\in \beta$, $(a,b)\in \alpha$ and $p\in
\POL\ari{2}\ab{A}$ with $p(a,u)\mathrel{\eta} p(a,v)$,
we have that $p(b, u) \mathrel{\eta} p(b, v)$.
To this end, 
let $p\in \POL\ari{2}\ab{A}$, 
$(u,v)\in\beta$, and $(a, b)\in\alpha$ with 
$(p(a,u), p(a,v))\in \eta$. 
Since $(a,b)\in \alpha$, there exist $p_a, p_b,\in \POL\ari{k}\ab{A}$
such that:
\begin{enumerate}
\item $\forall \vec{x}\in A^k\colon p_a(\vec{x})\mathrel{\mu}
p_b(\vec{x})$;
\item $\forall j\in \{3,\dots, n\}\colon
p_a(\vec{t}_j)=p_b(\vec{t}_j)$;
\item $p_a(\vec{t}_1)=a$ and $p_b(\vec{t}_1)=b$.
\end{enumerate}
Since $(u,v)\in \beta$,
\cite[Theorem~4.70(iii)]{McKMcnTay88}
yields
that there exist $q,q'\in \POL\ari{k}\ab{A}$ such that
$q(\vec{t}_1)=u$, $q(\vec{t}_2)=v$, $q'(\vec{t}_1)=v$ and
$q'(\vec{t}_2)=u$. We define $p_u,p_v\colon A^k \to A$ by letting
\[
\begin{split}
p_u(\vec{x})&=d(q(\vec{x}),q'(\vec{x}), v) \text{ and}\\
p_v(\vec{x})&= d(q(\vec{x}),u,v)
\end{split}
\]
for all $\vec{x}\in A^k$.
Thus, we have that $p_u,p_v\in \POL\ari{k}\ab{A}$, $p_u(\vec{t}_1)=u$,
$w:=p_u(\vec{t}_2)=d(v,u,v)=p_v(\vec{t}_2)$ and
$p_v(\vec{t}_1)=v$.
We further define $h,g\in\POL\ari{k}\ab{A}$ for all $\vec{x}\in
A^k$ by
\begin{align*}
g(\vec{x})&=  p(p_b(\vec{x}),p_u(\vec{x})),\\
h(\vec{x})    &=d(g(\vec{x}),\,
                  d(p(p_a(\vec{x}), p_v(\vec{x})),
                    p(p_a(\vec{x}), p_u(\vec{x})),
                    g(\vec{x})),\,
                  p(p_b(\vec{x}), p_v(\vec{x}))).
\end{align*}
We have for all $j\in\set{3, \dotsc, n}$ that
$p(p_a(\vec{t}_j),p_u(\vec{t}_j))=g(\vec{t}_j)$ and hence
\begin{align*}
h(\vec{t}_j)&=d(g(\vec{t}_j),\,
                d(p(p_a(\vec{t}_j), p_v(\vec{t}_j)),
                  p(p_a(\vec{t}_j), p_u(\vec{t}_j)),
                  g(\vec{t}_j)),\,
                p(p_b(\vec{t}_j), p_v(\vec{t}_j)))\\
            &=d(g(\vec{t}_j),\,
                d(p(p_a(\vec{t}_j), p_v(\vec{t}_j)),
                  g(\vec{t}_j),
                  g(\vec{t}_j)),\,
                p(p_b(\vec{t}_j), p_v(\vec{t}_j)))\\
            &=d(g(\vec{t}_j),\,
                p(p_a(\vec{t}_j), p_v(\vec{t}_j)),\,
                p(p_b(\vec{t}_j), p_v(\vec{t}_j)))\\
            &=d(g(\vec{t}_j),\,
                p(p_b(\vec{t}_j), p_v(\vec{t}_j)),\,
                p(p_b(\vec{t}_j), p_v(\vec{t}_j)))
             =g(\vec{t}_j).
\end{align*}
Moreover, since $w=p_u(\vec{t}_2)=p_v(\vec{t}_2)$, we have
\begin{align*}
h(\vec{t}_2)&=d(g(\vec{t}_2),\,
                d(p(p_a(\vec{t}_2), p_v(\vec{t}_2)),
                  p(p_a(\vec{t}_2), p_u(\vec{t}_2)),
                  g(\vec{t}_2)),\,
                p(p_b(\vec{t}_2), p_v(\vec{t}_2)))\\
            &=d(g(\vec{t}_2),\,
                d(p(p_a(\vec{t}_2), w),
                  p(p_a(\vec{t}_2), w),
                  g(\vec{t}_2)),\,
                p(p_b(\vec{t}_2), w))\\
            &=d(g(\vec{t}_2),\,
                g(\vec{t}_2),\,
                p(p_b(\vec{t}_2), w))\\
            &=p(p_b(\vec{t}_2), w)
             =p(p_b(\vec{t}_2), p_u(\vec{t}_2))
             =g(\vec{t}_2).
\end{align*}
For every $\vec{x}\in A^k$ we have
$p_a(\vec{x})\mathrel{\mu}p_b(\vec{x})$,
and hence 
\[p(p_a(\vec{x}),p_u(\vec{x}))\mathrel{\mu}p(p_b(\vec{x}),p_u(\vec{x})),\]
and therefore
$p(p_a(\vec{x}),p_u(\vec{x}))\mathrel{\mu}g(\vec{x})$.
Consequently,
\begin{multline*}
d(p(p_a(\vec{x}),p_v(\vec{x})),p(p_a(\vec{x}),p_u(\vec{x})),g(\vec{x}))
\mathrel{\mu}
d(p(p_b(\vec{x}),p_v(\vec{x})),g(\vec{x}),g(\vec{x}))\\
=p(p_b(\vec{x}),p_v(\vec{x})),
\end{multline*}
and hence
\begin{align*}
h(\vec{x})&=  d(g(\vec{x}),
                d(p(p_a(\vec{x}), p_v(\vec{x})),
                  p(p_a(\vec{x}), p_u(\vec{x})),
                  g(\vec{x})),
                p(p_b(\vec{x}), p_v(\vec{x})))\\
&\mathrel{\mu}d(g(\vec{x}),
                p(p_b(\vec{x}), p_v(\vec{x})),
                p(p_b(\vec{x}), p_v(\vec{x})))
=g(\vec{x}).
\end{align*}
From this we deduce that $h(\vec{t}_1)\mathrel{\eta}g(\vec{t}_1)$,
and thus, by applying the unary polynomial
$z\mapsto d(p(b,u),d(p(a,v),z,p(b,u)),p(b,v))$ to the pair
$(p(a,v),p(a,u))\in\eta$, we have
\begin{align*}
p(b,v)
&=d(p(b, u), p(b, u), p(b, v))\\
&=d(p(b, u), d(p(a, v), p(a, v),p(b, u)), p(b, v))\\
&\mathrel{\eta}d(p(b, u), d(p(a, v), p(a, u),p(b, u)), p(b, v))\\
&=h_1(\vec{t}_1)\mathrel{\eta} h_2(\vec{t}_1)=p(b, u).
\end{align*}
Hence $p(b, u)\mathrel{\eta} p(b, v)$.
Thus, by the definition of commutator we have 
 $[\alpha,\beta]\leq \eta$.
We now split the proof of $\alpha=\eta$ into two cases.

\textbf{Case 1}: $[\mu,\mu]=\mu$:
In this case, Proposition~\ref{prop:typesandhomogenuity}
yields that $\mu$ is an atom. 
By definition we have 
$\bottom{A}\leq \eta\leq \alpha\leq\mu$, and therefore,
if $\alpha=\bottom{A}$, then 
$\eta=\alpha$. 
Next, let us consider the case that
$\alpha=\mu$. 
Proposition~\ref{prop:homogeneueselementsplittingwithcentralizer}
implies that $\beta\geq \mu$, hence since $\alpha=\mu$ is not abelian,
\cite[Proposition~2.5]{Aic06} and the monotonicity of the commutator yield
\[\alpha=\mu=[\mu,\mu]\leq[\mu,\beta]=[\alpha,\beta]\leq\eta\leq\alpha.\]
Thus, $\alpha=\eta$.

\textbf{Case 2}: $[\mu,\mu]< \mu$:
We show that $[\alpha,\beta]=\alpha$. Seeking a contradiction, 
let us assume that $[\alpha,\beta]< \alpha$. Then 
there exists $\delta\in \Con\ab{A}$ such that 
$[\alpha,\beta]\leq\delta\prec\alpha\leq \mu$.
Hence 
Proposition~\ref{prop:typesandhomogenuity}\eqref{item:typetwocasehomogeneous} 
implies that $(\delta:\alpha)=(\Phi(\mu):\mu)$.
Moreover, \cite[Proposition~2.5]{Aic06} implies
$[\beta,\alpha]=[\alpha,\beta]\leq\delta$, and
therefore $\beta\leq (\delta:\alpha)=(\Phi(\mu):\mu)$. 
This yields a contradiction with the choice of $\vec{t}_1$ and $\vec{t}_2$. 
Thus, $[\alpha, \beta]=\alpha$, and therefore, $\alpha\leq\eta$. 
Since the definition of~$\alpha$ and~$\eta$ yields $\eta\leq \alpha$,
we conclude that $\eta=\alpha$. 

This concludes the proof that $\alpha=\eta$.
 
By assumption there are $p,q\in\POL\ari{k}\ab{A}$ with image contained in~$U$,
such that $p$ interpolates~$l$ at $\{\vec{t}_2,\dots , \vec{t}_n\}$, 
and~$q$ interpolates~$l$ at $\{\vec{t}_1,\vec{t}_3,\dots , \vec{t}_n\}$.
Since $U^2\subseteq \mu$, we have that
$p(\vec{x})\mathrel{\mu} q(\vec{x})$ for all $\vec{x}\in A^k$,
and that $p(\vec{t}_i)= l(\vec{t}_i)=q(\vec{t}_i)$ for every
$i\in\{3,\dots, n\}$. Hence $(p(\vec{t}_1), q(\vec{t}_1))\in\alpha$.
Since $\alpha=\eta$ and $q(\vec{t}_1)= l(\vec{t}_1)$, we have that
$(p(\vec{t}_1), l(\vec{t}_1))\in\eta$.
Therefore, there exist $p_2, p_3\in \POL\ari{k}\ab{A}$ such that:
\begin{enumerate}
\item $\forall i\in\{2,\dots, n\}\colon
p_2(\vec{t}_i)=p_3(\vec{t}_i)$;
\item $\forall \vec{x}\in A^k\colon p_2(\vec{x})\mathrel{\mu}
p_3(\vec{x})$;
\item $p_2(\vec{t}_1)=p(\vec{t}_1)$ and $p_3(\vec{t}_1)=l(\vec{t}_1)$.
\end{enumerate}
Let $p_T\colon A^k\to A$ be defined by
$p_T(\vec{x})=d(p(\vec{x}),p_2(\vec{x}),p_3(\vec{x}))$ for all
$\vec{x}\in A^k$.
Clearly, $p_T\in\POL\ari{k}\ab{A}$. Moreover, we have that for all
$i\in\{2,\dots, n\}$
\[p_T(\vec{t}_i)=
d(p(\vec{t}_i), p_2(\vec{t}_i),
p_3(\vec{t}_i))=p(\vec{t}_i)=l(\vec{t}_i).\]
Furthermore,
\[p_T(\vec{t}_1)=
d (p(\vec{t}_1), p_2(\vec{t}_1),
p_3(\vec{t}_1))=d(p(\vec{t}_1),p(\vec{t}_1),
l(\vec{t}_1))=l(\vec{t}_1).
\]
Moreover, we have that for all $\vec{x}\in A^k$
\[p_T(\vec{x})=
d(p(\vec{x}),p_2(\vec{x}), p_3(\vec{x}))\,\mu \, d(o,p_2(\vec{x}),
p_2(\vec{x}))=o.
\]
Thus, $p_T(\vec{x})\in o/\mu=U$ and we can conclude that $p_T$ has
target~$U$ and interpolates~$l$ on~$T$.
\end{proof}

\begin{proposition}\label{prop:interpolation_in_a_class_nonAbelia_case}
Let $\ab{A}$ be a Mal'cev algebra whose congruence lattice has finite
height,
let $\mu$ be a 
homogeneous non-abelian element of $\Con\ab{A}$,
let $o\in A$, and let $k\in\N$. Then for every finite subset $T$
of $A^k$,
and for all $l\colon T\to o/\mu$ the following are equivalent:
\begin{enumerate}
\item $l$ is congruence-preserving;\label{item:compatibel:in:prop:interpolation_in_a_class_nonAbelia_case}
\item there exists $p_T\in\POL\ari{k}\ab{A}$ such that for all $\vec{t}\in T$
we have $p_T(\vec{t})=l(\vec{t})$ and $p_T(A^k)\subseteq o/\mu$. \label{item:interpolable:in:prop:interpolation_in_a_class_nonAbelia_case}
\end{enumerate}
\end{proposition}
\begin{proof}
We let $U:=o/\mu$. 
Proposition~\ref{prop:typesandhomogenuity}\eqref{item:typethreecasehomogeneous} yields
that $(\bottom{A}:\mu)=\mu^*$, and that $\mu$ is an atom of $\Con\ab{A}$. 
Thus, 
\begin{equation}\label{eq:equazione_centralizzatore_uguale_mu_star}
\Phi(\mu)=\bottom{A}\text{ and }(\Phi(\mu):\mu)=\mu^*.
\end{equation} 
Thus, Proposition~\ref{prop:homogeneueselementsplittingwithcentralizer} implies that 
\begin{equation}\label{eq:equazione_split}
\forall \alpha\in\Con\ab{A}: \alpha\geq \mu \text{ or }
\alpha\leq (\bottom{A}:\mu)=\mu^*.
\end{equation}

The implication \eqref{item:interpolable:in:prop:interpolation_in_a_class_nonAbelia_case} $\Rightarrow$ 
\eqref{item:compatibel:in:prop:interpolation_in_a_class_nonAbelia_case} is trivial. 
We prove the other implication by induction on $\card{T}$. 
\par
\textbf{BASE}: \emph{$\card{T}\leq 2$}: We split the base of the induction into three cases.
\par
\textbf{Case 1}: \emph{$\card{l(T)}=1$}: 
Let $a$ be the unique element of $l(T)$. Then the polynomial function with
constant value $a$ interpolates $l$ on $T$ and has image contained in $o/\mu$. 
\par 
\textbf{Case 2}: \emph{$\card{T}=2$ and $\card{T/(\Phi(\mu):\mu)}=1$}: We show that in this case 
$l(\vec{t}_1)=l(\vec{t}_2)$, and therefore the polynomial function with constant value
$l(\vec{t}_1)$ satisfies the required properties. To this end, 
we assume that $\card{T}=2$ and 
$\vec{t}_1\equiv_{(\bottom{A}:\mu) }\vec{t}_2$, and prove that $l(\vec{t}_1)=l(\vec{t}_2)$. 
Equation \eqref{eq:equazione_centralizzatore_uguale_mu_star}
implies that $(\bottom{A}:\mu)=\mu^*$. Hence the compatibility of $l$ implies that 
since $\vec{t}_1\equiv_{(\bottom{A}:\mu)}\vec{t}_2$, we have
$l(\vec{t}_1)\,\mu^* \, l(\vec{t}_2)$.
Since $l(T)\subseteq o/\mu$, we have that $l(\vec{t}_1)\mathrel{\mu} l(\vec{t}_2)$. 
Thus, Lemma~\ref{lemma:whymumeetmustariszero} implies that
$l(\vec{t}_1)=l(\vec{t}_2)$. 
\par
\textbf{Case 3}: 
\emph{$\card{T}=2$, $l(\vec{t}_1)\neq l(\vec{t}_2)$, and $\vec{t}_1\not\equiv_{(\bottom{A}:\mu)}\vec{t}_2$}:
Let $l(\vec{t}_1)=a_1$ and $l(\vec{t}_2)=a_2$. Since~$\mu$ is not abelian,
\cite[Proposition~2.3]{Aic06}
implies that there exist
$a,b,u,v\in A$ and $t\in \POL\ari{2}\ab{A}$ such that 
\[a\mathrel{\mu}b, \quad u\mathrel{\mu} v,
\quad  t(a,u)=t(a,v),\quad t(b,u)\neq t(b,v).\] 
Therefore, $\vec{t}_1\not\equiv_{(\bottom{A}:\mu)}\vec{t}_2$ and 
\eqref{eq:equazione_split} imply that
\[
(u,v)\in\mu\leq
\Cg{\{(\vec{t}_1(1),\vec{t}_2(1)),\dots,
(\vec{t}_1(k),\vec{t}_2(k))\}}.
\]
Then,
\cite[Theorem~4.70(iii)]{McKMcnTay88}
yields that there is $h\in\POL\ari{k}\ab{A}$ such that
$h(\vec{t}_1)=u$ and $h(\vec{t}_2)=v$.

We now focus on $\Cg{\{(t(b,u),t(b,v))\}}$. We have that 
$u\mathrel{\mu} v$, hence 
$t(b,u)\mathrel{\mu}t(b,v)$ and therefore
$\Cg{\{(t(b,u),t(b,v))\}}\leq \mu$. 
Since $t(b,u)\neq t(b,v)$, and since $\mu$ is an atom of $\Con\ab{A}$,
we have 
$\Cg{\{(t(b,u),t(b,v))\}}= \mu$.
Since
$(a_1,a_2)\in U^2\subseteq\mu=
\Cg{\{(t(b,u),t(b,v))\}}$,
\cite[Theorem~4.70(iii)]{McKMcnTay88}
yields
that there exists $p\in \POL\ari{1}\ab{A}$ such that 
\[p(t(b,u))=a_1\, \text{and}\, p(t(b,v))=a_2.\]
Let us define the $k$-ary polynomial $p_T\colon A^k\to A$ for all
$\vec{z}\in A^k$ by \[p_T(\vec{z}):=
p(d(t(b,h(\vec{z})),t(a,h(\vec{z})),t(a,u))).\]
Then for any $\vec{x}\in A^k$ such that $h(\vec{x})=u$, we have
\begin{equation}\label{eq:pT-is-lt1-on-hx=u}
p_T(\vec{x})=p(d(t(b,u),t(a,u),t(a,u)))=p(t(b,u))=a_1=l(\vec{t}_1).
\end{equation}
In particular, this holds for $\vec{x}=\vec{t}_1$.
Moreover, since $t(a,u)=t(a,v)$, we have
\[
p_T(\vec{t}_2)=p(d(t(b,v),t(a,v),t(a,u)))=p(t(b,v))=a_2=l(\vec{t}_2).
\]
For all $\vec{x}\in A^k$,
let $\alpha_\vec{x}:= \Cg{\{(u, h(\vec{x}))\}}$.
We have that for all $\vec{x}\in A^k$
\[(b, h(\vec{x}))\equiv_\mu (a, h(\vec{x}))\equiv_{\alpha_{\vec{x}}}
(a,u).\]
Thus, Lemma~\ref{lemma:ex_citation_to_prop2.6Aic06} implies that for all $\vec{x}\in A^k$
\begin{equation}\label{eq:equations_affine_modulo_commutators_in_interpolating_proposition}
d(t(b, h(\vec{x})),t(a, h(\vec{x})),t(a,u))\,[\mu, \alpha_{\vec{x}}]\,
t(d(b,a,a), d(h(\vec{x}),h(\vec{x}),u))=t(b,u).
\end{equation}
Since $[\mu, \alpha_{\vec{x}}]\leq \mu$,
condition~\eqref{eq:equations_affine_modulo_commutators_in_interpolating_proposition}
yields that for all $\vec{x}\in A^k$
\[
d(t(b, h(\vec{x})),t(a, h(\vec{x})),t(a,u))\mathrel{\mu} t(d(b,a,a),
d(h(\vec{x}),h(\vec{x}),u))=t(b,u).
\]
Thus, for all $\vec{x}\in A^k$ we have
\begin{align*}
p_T(\vec{x})&=p(d(t(b,h(\vec{x})), t(a, h(\vec{x})), t(a,u)))\\
&\equiv_\mu  p(t(d(b,a,a),d(h(\vec{x}),h(\vec{x}),u)))\\
&=p(t(b,u))=a_1=l(\vec{t}_1)\equiv_\mu o.
\end{align*}
Hence $p_T(\vec{x})\in U$ for all $\vec{x}\in A^k$.

\textbf{INDUCTION STEP}: 
Let $\card{T}=n\geq 3$, and let us
assume that for each $\vec{t}\in T$ the function $l$ can be interpolated on 
$T\setminus\{\vec{t}\}$
by a polynomial whose image is a subset of~$U$. We prove that~$l$ can be
interpolated on~$T$ by a polynomial whose image is a subset of $U$.
We now split the proof into two cases.
\par
\textbf{Case 1}: 
\emph{For all $\vec{s},\vec{t}\in T$
we have that 
\[
\Cg{\{(\vec{s}(1),\vec{t}(1)),\dots, (\vec{s}(k),\vec{t}(k))\}}\leq (\Phi(\mu):\mu).
\]}
By \eqref{eq:equazione_centralizzatore_uguale_mu_star} we have $(\Phi(\mu):\mu)=\mu^*$.
Since $l$ is congruence-preserving,
for all $\vec{s},\vec{t}\in T$ we have that
$l(\vec{s})\mathrel{\mu^*}l(\vec{t})$.
Moreover, since $f(T)\subseteq U$, we have that
$(l(\vec{s}),l(\vec{t}))\in \mu\wedge \mu^*=\bottom{A}$.
Hence $\card{l(T)}=1$, and the polynomial with constant value $l(\vec{t}_1)$
interpolates $l$ on $T$ and has image contained in $U$. 

\textbf{Case 2}: \emph{There exist $\vec{t}_1, \vec{t}_2\in T$
such that\[
\Cg{\{(\vec{t}_1(1),\vec{t}_2(1)),\dots, (\vec{t}_1(k),\vec{t}_2(k))\}}\nleq (\Phi(\mu):\mu).
\]}
Then, we can apply 
Lemma~\ref{lemma.come_si_fanno_le_induzioni}
and conclude that $l$ can be interpolated on $T$ by a polynomial 
function of $\ab{A}$ whose image is a subset of $U$. 
\end{proof}
For a congruence $\alpha$ of a (Mal'cev) algebra $\ab{A}$ and for 
$k\in\N$ we define
\[
\alpha^k:=\Biggl\{\begin{pmatrix}
\vec{a}\\
\vec{b}
\end{pmatrix}\in (A^k)^2\mid 
\forall i\leq k\colon (\vec{a}(i),\vec{b}(i))\in\alpha
\Biggr\}.
\]
\begin{proposition}\label{prop:interpolating_on_cosets_abelian_case}
Let $\ab{A}$ be a finite Mal'cev algebra with (SC1), let $\mu$ be a
homogeneous abelian element of $\Con\ab{A}$,
let $o\in A$, let $k\in\N$, let $T$ be a
finite subset of $A^k$, 
and let $f\colon T\to o/\mu$. Then the following are equivalent:
\begin{enumerate}
\item There exists $p_T\in\POL\ari{k}\ab{A}$ with $p_T(\vec{t})=f(\vec{t})$ for all 
$\vec{t}\in T$, and $p_T(A^k)\subseteq o/\mu$;\label{item:interpolating_polynomial_prop:interpolating_on_cosets_abelian_case}
\item for each $(\Phi(\mu):\mu)^k$-equivalence class of the form 
$\vec{v}/(\Phi(\mu):\mu)^k$ there exists $p_{\vec{v}}\in\POL\ari{k}\ab{A}$
such that $p_{\vec{v}}(\vec{t})=f(\vec{t})$ for all 
$\vec{t}\in T\cap (\vec{v}/(\Phi(\mu):\mu)^k)$. \label{item:interpolation-on_cosets_prop:interpolating_on_cosets_abelian_case}
\end{enumerate}   
\end{proposition}
\begin{proof}
The implication 
\eqref{item:interpolating_polynomial_prop:interpolating_on_cosets_abelian_case}
$\Rightarrow$
\eqref{item:interpolation-on_cosets_prop:interpolating_on_cosets_abelian_case}
is trivial.
Next, we prove the opposite implication by induction on $\card{T}$. 
For each pair $\vec{t}_1, \vec{t}_2$ of elements of $T$ we let 
\[
\vartheta_{\ab{A}}(\vec{t}_1, \vec{t}_2):=\Cg{\{(\vec{t}_1(1),\vec{t}_2(1)),\dots, (\vec{t}_1(k),\vec{t}_2(k))\}}.
\]
\par
\textbf{BASE}: \emph{$\card{T}\leq 2$}: We split the base into two cases.
\par
\textbf{Case 1}: \emph{$\card{T/(\Phi(\mu):\mu)^k}=1$}: Then 
\eqref{item:interpolation-on_cosets_prop:interpolating_on_cosets_abelian_case} implies that
there exists $p\in\POL\ari{k}\ab{A}$ such that 
for all $\vec{t}\in T$ we have $f(\vec{t})=p(\vec{t})$. Let $e_\mu^o$
be the idempotent polynomial function constructed in 
Theorem~\ref{teor:existence_of_idempotent_class_homogeneous}. 
Then we have $p_T:=e_\mu^o\circ p$ satisfies the requirements of 
\eqref{item:interpolating_polynomial_prop:interpolating_on_cosets_abelian_case}. 
\par
\textbf{Case 2}: \emph{There exist $\vec{t}_1,\vec{t}_2\in T$ with $(\vec{t}_1,\vec{t}_2)\notin (\Phi(\mu):\mu)^k$}: 
If $(\vec{t}_1,\vec{t}_2)\notin (\Phi(\mu):\mu)^k$ then $\vartheta_{\ab{A}}(\vec{t}_1, \vec{t}_2)\nleq (\Phi(\mu):\mu)$.
Thus, by Proposition~\ref{prop:homogeneueselementsplittingwithcentralizer},
$\vartheta_{\ab{A}}(\vec{t}_1, \vec{t}_2)\geq \mu$.
Since $f(\vec{t}_1)\mathrel{\mu}f(\vec{t}_2)$, we have 
$(f(\vec{t}_1),f(\vec{t}_2))\in\vartheta_{\ab{A}}(\vec{t}_1, \vec{t}_2)$. 
Thus, \cite[Theorem~4.70(iii)]{McKMcnTay88}
yields that there exists $h\in\POL\ari{k}\ab{A}$
such that $f(\vec{t}_1)=h(\vec{t}_1)$ and $f(\vec{t}_2)=h(\vec{t}_2)$. 
Let $e_\mu^o$
be the idempotent polynomial function constructed in 
Theorem~\ref{teor:existence_of_idempotent_class_homogeneous}.
Thus, $p_T:=e_\mu^o\circ h$ satisfies the requirements of
\eqref{item:interpolating_polynomial_prop:interpolating_on_cosets_abelian_case}.
\par
\textbf{INDUCTION STEP}: We split the induction step into two cases.
\par
\textbf{Case 1}: \emph{There exist $\vec{t}_1,\vec{t}_2\in T$ 
such that $\vartheta_{\ab{A}}(\vec{t}_1, \vec{t}_2)\nleq (\Phi(\mu):\mu)$}:
By induction hypothesis we have that $\card{T}\geq 3$ and
we can interpolate $f$ with a polynomial whose image is contained 
in $o/\mu$ on every subset of $T$ of cardinality $\card{T}-1$. 
Then we can apply Lemma~\ref{lemma.come_si_fanno_le_induzioni}
and obtain a polynomial with image contained in $o/\mu$
that interpolates $f$ on $T$.
\par
\textbf{Case 2}: \emph{For all $\vec{t}_1,\vec{t}_2\in T$ we have that 
$\vartheta_{\ab{A}}(\vec{t}_1,\vec{t}_2)\leq (\Phi(\mu):\mu)$}:
Then $\card{T/(\Phi(\mu):\mu)^k}=1$ and we can argue as in the 
first case of the induction base. 
\end{proof}

\section{Finite Mal'cev algebras with (SC1) and (AB2) are strictly 1-affine complete}\label{sec:proof_theorem_strict_affine_compl}
In this section we prove 
Theorem~\ref{teor.main_theorem_sc1_ab2_imply_compl}.
\begin{proposition}\label{prop:qui_si_interpola_nel_caso_abeliano_le_funzioni_compatibili_definite_sui_coset_del_centralizzatore}
Let $\ab{A}$ be a finite Mal'cev algebra
with (SC1), let $o, v\in A$, 
let $\mu$ be an homogeneous
abelian element of $\Con\ab{A}$ such that 
$(\ab{A}, \mu)$ satisfies (AB2),
let $T\subseteq v/(\bottom{A}:\mu)$, and let $f\colon
T\to o/\mu$ be a congruence-preserving function. 
Then $f$ can be interpolated on $T$
by a polynomial whose image is contained in $o/\mu$. 
\end{proposition}
\begin{proof}
First, we observe that 
if $\card{o/\mu}=1$, then the constant polynomial function 
with constant value $o$ interpolates $f$ on $T$ and has 
image contained in $o/\mu$. 
In the remainder of the proof we consider the case that 
$\card{o/\mu}>1$. 
Let us define $f_1\subseteq (v/\mu)\times (o/\mu)$ as follows:
\[
f_1=\{(u, f(d(u,v,u^*))\mid u\in v/\mu\text{ and }\exists u^*\in v/\mu^*\colon d(u,v,u^*)\in T\}.
\]
Next, we prove that $f_1$ is functional. To this end,
let $u\in v/\mu$ and let $a^*, b^*\in v/\mu^*$ be such that
$\{d(u,v,a^*),d(u,v,b^*)\}\subseteq T$. 
Then, since $f$ is congruence-preserving, 
we have $f(d(u,v,a^*))\mathrel{\mu^*}f(d(u,v,b^*))$,
and since $f(T)\subseteq o/\mu$, we have
$f(d(u,v,a^*))\mathrel{\mu\wedge\mu^*}f(d(u,v,b^*))$.
Hence Lemma~\ref{lemma:whymumeetmustariszero} yields 
$f(d(u,v,a^*))=f(d(u,v,b^*))$.

Next, we show that $f_1$ is congruence-preserving. 
To this end, let $\theta\in\Con\ab{A}$, and 
let $(u_1, u_2)\in \theta$ with $u_1,u_2$ in
the domain of $f_1$. We show that $f_1(u_1)\mathrel{\theta}f_1(u_2)$.
Let us define $\tilde{\theta}:=\Cg{\{(u_1, u_2)\}}$. 
Since $u_1$ and $u_2$ are in the domain of $f_1$,
there exist $u_1^*, u_2^*\in A$ such that $d(u_1, v, u_1^*)$ and
$d(u_2, v, u_2^*)$ belong to $T$. 
Moreover, we have 
$f_1(u_1)=f(d(u_1, v, u_1^*))$ and $f_1(u_2)=f(d(u_2, v, u_2^*))$.
Furthermore, we have 
\[
d(u_1, v, u_1^*)\mathrel{\tilde{\theta}}d(u_2, v, u_1^*)\mathrel{\mu^*}d(u_2, v, u_2^*).
\]
Thus, $d(u_1, v, u_1^*)\mathrel{\tilde{\theta}\vee\mu^*}d(u_2, v, u_2^*)$,
and therefore, since $f$ is congruence-preserving, we have
$f_1(u_1)\mathrel{\tilde{\theta}\vee\mu^*}f_1(u_2)$. 
Since $f(T)\subseteq o/\mu$ we infer that 
$f_1(u_1)\mathrel{(\tilde{\theta}\vee\mu^*)\wedge\mu}f_1(u_2)$.
Since $\tilde{\theta}\leq\mu$ and $\Con\ab{A}$ is modular,
Lemma~\ref{lemma:whymumeetmustariszero} yields
\[(\tilde{\theta}\vee\mu^*)\wedge\mu=\tilde{\theta}\vee(\mu^*\wedge \mu)=
\tilde{\theta}\vee\bottom{A}=\tilde{\theta}\subseteq \theta.\]
Thus, $f_1(u_1)\mathrel{\theta}f_1(u_2)$.

Next, we show that there exists $p\in\POL\ari{1}\ab{A}$
that interpolates $f_1$ on its domain and
has image contained in $o/\mu$.
We split the proof into two cases:\\
\textbf{Case 1}: \emph{$\card{v/\mu}=1$}: In this case
the constant function with constant value $f_1(v)$ 
interpolates $f_1$ on its domain and has image contained in
$o/\mu$. \\
\textbf{Case 2}: \emph{$\card{v/\mu}>1$}: In this case 
$o/\mu$ and $v/\mu$ are $\quotienttct{\bottom{A}}{\mu}$-minimal
sets by Theorem~\ref{teor:ABp_Implies_polynomial_equivalence}
and therefore, $o/\mu$ and $v/\mu$ are
polynomially isomorphic by \cite[Theorem~2.8(1)]{HobbMcK}.
Hence there exist $t_{v \rightarrow o}, t_{o \rightarrow v}\in \POL\ari{1}\ab{A}$,
such that $t_{v \rightarrow o}\restrict{v/\mu}$ is 
an isomorphism between
$\ab{A}\restrict{v/\mu}$  
and $\ab{A}\restrict{o/\mu}$, and
$(t_{v \rightarrow o}\restrict{v/\mu})^{-1}=
t_{o \rightarrow v}\restrict{o/\mu}$. 
Let $T_1$ be the domain of $f_1$, and let 
$f_2:=t_{o\rightarrow v}\circ f_1$.
Then $f_2$ is a partial 
congruence-preserving function 
on $\ab{A}\restrict{v/\mu}$. 
By Proposition~\ref{prop:homogenuityandcolpementation}\eqref{item:simplecomplmodlatticebelowhom}
the interval $\interval{\bottom{A}}{\mu}$ is
a simple complemented modular lattice. 
Since, by Corollary~\ref{cor:introducing_ABp_field_abelian_congruences}
and Lemma~\ref{lemma:piu_meno_fanno_modulo_su-anello_polinomi_ristretto_nella_classe_congruenza},
$\ab{A}\restrict{v/\mu}$ is polynomially equivalent 
to a $h$ dimensional vector space over $\GF{2}$,
by \cite[Proposition~8.2]{AicIdz04}
there exists a unary polynomial function $q$ 
such that $q(x)=f_2(x)$ for every $x\in T_1$.
Then for each $x\in T_1$ we have that
\[
t_{v\rightarrow o} (q(x))=t_{v\rightarrow o} (f_2(x))=t_{v\rightarrow o}( t_{o\rightarrow v} (f_1(x)))=f_1(x).
\]
Thus, $t_{v\rightarrow o} \circ q$ interpolates $f_1$ on $T_1$.
Let $e_\mu^o$ be the idempotent polynomial constructed in
Theorem~\ref{teor:existence_of_idempotent_class_homogeneous}.
Then $e_\mu^o\circ t_{v\rightarrow o}\circ q$ interpolates 
$f_1$ on its domain 
and has image contained in $o/\mu$.

Thus, we can conclude that
there exists a polynomial function $p$
that interpolates $f_1$ on its domain 
and has image contained in $o/\mu$. 
Next, we show that for all $t\in T$ we have $f(t)=p(t)$.
To this end, let us fix $t\in T$. 
Since Lemma~\ref{lemma:whymumeetmustariszero}
and 
Proposition~\ref{prop:if_an_homogeneues_congruence_is_not_neutral_then_it_is_abelian}
imply
$\interval{\bottom{A}}{\mu}
\nearrow \interval{\mu^*}{(\bottom{A}:\mu)}$,
Lemma~\ref{lemma:conseguenze_permutabilita_intervalli_proiettivi_esistenza,d(b,o,c)}
yields that 
there exist $u,u^*\in A$ such that $t=d(u, v,u^*)$, $u\in v/\mu$,
and $u^*\in v/\mu^*$. 
Thus, we have 
\[
f(t)=f(d(u,v,u^*))=f_1(u)=p(u).
\]
Since $u=d(u,v,v)$ and $u^*\mathrel{\mu^*} v$, we have 
$p(u)\mathrel{\mu^*}p(d(u,v,u^*))$,
and therefore, since $t\mathrel{\mu^*} u$, we have
$f(t)\mathrel{\mu^*}p(t)$. Since $f(t)\mathrel{\mu}o$
and $p(A)\subseteq o/\mu$,
we have that $f(t)\mathrel{\mu}p(t)$. Thus, 
Lemma~\ref{lemma:whymumeetmustariszero}
yields that $f(t)=p(t)$. 
\end{proof}

\begin{proposition}\label{prop:interpolating_partial_functions_inside_a_congruence_class_abelian_case_with_AB2}
Let $\ab{A}$ be a finite Mal'cev algebra with (SC1),
let $\mu$ be a homogeneous abelian congruence of $\ab{A}$
such that $(\ab{A},\mu)$ satisfies (AB2),
let $o\in A$, let $T\subseteq A$, and let $f\colon T\to o/\mu$
be a congruence-preserving function. Then there exists $p\in\POL\ari{1}\ab{A}$
whose image is contained in $o/\mu$
that interpolates $f$ on $T$. 
\end{proposition}
\begin{proof}
By Proposition~\ref{prop:the_centralizer_of_a_homogeneous_congruence}
and 
Proposition~\ref{prop:if_an_homogeneues_congruence_is_not_neutral_then_it_is_abelian},
we have that $\Phi(\mu)=\bottom{A}$ and  $(\bottom{A}:\mu)=(\Phi(\mu):\mu)=\mu\vee \mu^*$.
For each element $v/(\bottom{A}:\mu)$ of $T/(\bottom{A}:\mu)$
Proposition~\ref{prop:qui_si_interpola_nel_caso_abeliano_le_funzioni_compatibili_definite_sui_coset_del_centralizzatore}
yields a polynomial with image contained in $o/\mu$
that interpolates $f$ on $v/(\bottom{A}:\mu)$.
Thus, Proposition~\ref{prop:interpolating_on_cosets_abelian_case}
yields the desired polynomial function. 
\end{proof}

\begin{proposition}\label{prop:interpolation_on_quotient_yields_interpolation}
Let $\ab{A}$ be a finite Mal'cev algebra with (SC1),
let $\mu$ be a homogeneous congruence of $\ab{A}$
such that $(\ab{A}, \mu)$ satisfies (AB2),
let $T$ be a finite subset of $A$ and let $f\colon T\to A$
be a partial congruence-preserving function that can be interpolated
modulo $\mu$ by a polynomial function of $\ab{A}$. 
Then, $f$ can be interpolated by a polynomial 
function of $\ab{A}$. 
\end{proposition}
\begin{proof}
We first prove that every partial function that preserves
the congruences of $\ab{A}$
and has image contained in one $\mu$-class
can be interpolated by a polynomial function of $\ab{A}$
whose image is contained in the same $\mu$-class.
We split the proof into two cases:\\
\textbf{Case 1}: $[\mu,\mu]=\mu$: In this case we can apply 
Proposition~\ref{prop:interpolation_in_a_class_nonAbelia_case}
to construct the interpolating polynomial function.\\
\textbf{Case 2}: $[\mu,\mu]\neq\mu$: Then 
Proposition~\ref{prop:if_an_homogeneues_congruence_is_not_neutral_then_it_is_abelian}
yields that $\mu$ is abelian. 
Thus,
Proposition~\ref{prop:interpolating_partial_functions_inside_a_congruence_class_abelian_case_with_AB2}
yields that any congruence-preserving partial function whose domain is 
contained in one $\mu$-class
can be interpolated by a polynomial function
whose image is contained in the
same $\mu$-class. 

Thus, the assumptions of Lemma~\ref{lemma:fromcosetstotheinfinirtyandbeyond}
are satisfied, and $f$ can be interpolated by a polynomial function. 
\end{proof}
\begin{lemma}\label{lemma:SC1andAB2carryovertoquotients}
Let $\ab{A}$ be a  finite Mal'cev with at least two elements
that satisfies (AB2) and (SC1), and let $\mu$ be 
a homogeneous congruence of $\ab{A}$.
Then $\ab{A}/\mu$ satisfies (AB2) and (SC1).
\end{lemma}
\begin{proof}
Since (SC1) carries over to quotients we have that 
$\ab{A}/\mu$ satisfies (SC1). Furthermore, 
let $\alpha', \beta'\in \Con\ab{A}/\mu$ with
$\alpha'\prec\beta'$ and $[\beta', \beta']\leq \alpha'$.
Then there exists $\alpha, \beta\in \Con\ab{A}$
such that $\alpha'=\alpha/\mu$, and $\beta'=\beta/\mu$. 
Moreover, if $[\beta/\mu,\beta/\mu]\leq\alpha/\mu$
we have that $[\beta, \beta]\leq\alpha$
(cf.~\cite[Remark~4.6]{FreMcK87}).
Thus, each $\beta/\alpha$-class has at most two elements,
and therefore so does each $\beta'/\alpha'$-class.
\end{proof}
\begin{proof}[Proof of Theorem \ref{teor.main_theorem_sc1_ab2_imply_compl}]
We proceed by induction on $\card{A}$. If $\card{A}=1$, then 
$\ab{A}$ is strictly 1-affine complete.
Let us assume that $\card{A}>1$. 
Then Proposition~\ref{prop:SC1impliesAPMI} implies that $\ab{A}$ has (APMI);
Proposition~\ref{prop:homogenuityandcolpementation} yields that $\ab{A}$ 
has a homogeneous congruence $\mu$, and
Lemma~\ref{lemma:SC1andAB2carryovertoquotients}
implies that $\ab{A}/\mu$ satisfies (SC1) and 
(AB$2$). Thus, by the induction hypothesis $\ab{A}/\mu$ 
is strictly 1-affine complete.
Let $f\colon T\to A$ be a partial congruence-preserving function.
Since $\ab{A}/\mu$ is strictly 1-affine complete,
$f$ can be interpolated modulo $\mu$ by a polynomial
function of $\ab{A}$.
Furthermore, $(\ab{A}, \mu)$ satisfies (AB2).
Thus, 
Proposition~\ref{prop:interpolation_on_quotient_yields_interpolation} yields
that $f$ can be interpolated on $T$ by a polynomial function of $\ab{A}$.
\end{proof}

\section{Strictly 1-affine complete congruence regular Mal'cev algebras}\label{sec:concluding_remarks}
In this section we investigate necessary conditions 
for strict 1-affine completeness.
In particular, we will prove a  characterization 
of strictly 1-affine complete congruence regular
finite Mal'cev algebras, and 
of finite loops. 
We start with the following straightforward generalization of 
\cite[Proposition~4.5]{AicIdz04} from
expanded groups to Mal'cev algebras. 
\begin{proposition}\label{prop:AB2_necessaria}
Let $\ab{A}$ be a finite Mal'cev algebra, and 
let us assume that there exists a pair of congruences $\alpha, \beta$
with $\alpha\prec\beta$ and $[\beta,\beta]\leq \alpha$ and
there exists $o\in A$ such that  $\card{(o/\alpha)/(\beta/\alpha)}> 2$.
Then $\ab{A}$ is not strictly 1-affine complete.
\end{proposition}
\begin{proof}
First, we show that we can assume without loss of generality that 
$\beta$ is join irreducible and $\alpha=\beta^-$. 
To this end, let $\tilde{\beta}$ be minimal with the property $\tilde{\beta}\leq \beta$
and $\tilde{\beta}\nleq \alpha$. Then $\tilde{\beta}$ is join irreducible. 
Let $\tilde{\alpha}=\tilde{\beta}^-$.
Since $\interval{\tilde{\alpha}}{\tilde{\beta}}\nearrow\interval{\alpha}{\beta}$, 
Lemma~\ref{lemma:ismorphisms_between_classes_same_element_projective_intervals} yields
$\card{(o/\tilde{\alpha})/(\tilde{\beta}/\tilde{\alpha})}> 2$,
and Lemma~\ref{lemma:exAic18Lemma3.4}
yields that $[\tilde{\beta},\tilde{\beta}]\leq\tilde{\alpha}$.

Seeking a contradiction, let us assume that $\ab{A}$ is strictly
1-affine complete. 
By assumption, there exist $b_1, b_2\in A$ such that $b_1, b_2\in o/\beta$,
$b_1\notin o/\alpha$, $b_2\notin o/\alpha$, and $b_2\notin d(o,b_1, o)/\alpha$. 
Let us define $f\colon\{o,b_1, b_2, d(b_1, o,b_2)\}\to \{o, b_1\}$
by $\{o,b_1, b_2\}\mapsto o$ and $d(b_1, o,b_2)\mapsto b_1$.
Next we prove that $f$ is congruence-preserving. 
We first show that 
\begin{equation}\label{eq:equazione_chi_genera_beta_proposizion_AB2_necessaria}
\beta=\Cg{\{(o,b_1)\}}=\Cg{\{(o,b_2)\}}=\Cg{\{(b_2,d(o,b_1, o))\}}.
\end{equation}
To this end, let $x,y\in A$ be such that $(x,y)\in \beta\setminus\alpha$. 
Then $\Cg{\{(x,y)\}}\leq \beta$ and $\Cg{\{(x,y)\}}\nleq \alpha$.
Since $\alpha=\beta^-$ we infer that $\Cg{\{(x,y)\}}= \beta$.
Hence \eqref{eq:equazione_chi_genera_beta_proposizion_AB2_necessaria}
follows.

In the following we will show that $\beta$ is generated by some other pairs 
of elements from $A$ (cf.~identities~\eqref{eq:equatione_b1_somma_proposizione_AB2_necessaria}, 
\eqref{eq:equatione_b2_somma_proposizione_AB2_necessaria},
\eqref{eq:equatione_o_somma_proposizione_AB2_necessaria}). This will imply that 
$f$ is a congruence-preserving unary partial function.

By Lemma~\ref{lemma:piu_emeno_fanno_gruppoo_abeliano_nella_classe_congruenza},
$(o/\alpha)/(\beta/\alpha)$ is an abelian group 
with respect to the operation 
$x/\alpha +y/\alpha= d(x,o,y)/\alpha$.
Next, we show that 
\begin{equation}\label{eq:equatione_b1_somma_proposizione_AB2_necessaria}
\beta=\Cg{\{(b_1, d(b_1, o,b_2))\}}.
\end{equation} 
Seeking a contradiction, let us assume 
that $\beta\neq\Cg{\{(b_1, d(b_1, o,b_2))\}}$. Then 
$(b_1, d(b_1, o,b_2))\in \alpha$ and therefore,
$b_1/\alpha=d(b_1, o,b_2)/\alpha= b_1/\alpha +b_2/\alpha$.
This implies that $b_2/\alpha=o/\alpha$ in contradiction with the
assumption $b_2\notin o/\alpha$. 
In a similar way one can prove that 
\begin{equation}\label{eq:equatione_b2_somma_proposizione_AB2_necessaria}
\beta=\Cg{\{(b_2, d(b_1, o,b_2))\}}.
\end{equation} 
Next, we show that
\begin{equation}\label{eq:equatione_o_somma_proposizione_AB2_necessaria}
\beta=\Cg{\{(o, d(b_1, o,b_2))\}}.
\end{equation} 
Seeking a contradiction, let us assume that 
$\beta\neq\Cg{\{(o, d(b_1, o,b_2))\}}$. Then 
$(o, d(b_1, o,b_2)\in \alpha$, and therefore, 
$o/\alpha= b_1/\alpha +b_2/\alpha$. Hence
$b_2/\alpha= -b_1/\alpha=d(o,b_1, o)/\alpha$, 
in contradiction with the assumption 
$b_2\notin d(o,b_1, o)/\alpha$. 

Since $(o,b_1)\in\beta$,
\eqref{eq:equatione_b1_somma_proposizione_AB2_necessaria},
\eqref{eq:equatione_b2_somma_proposizione_AB2_necessaria},
and 
\eqref{eq:equatione_o_somma_proposizione_AB2_necessaria} yield that
$f$ is congruence-preserving. 
Since $\ab{A}$ is strictly 1-affine complete, there exists $p\in\POL\ari{1}\ab{A}$
such that $p(\{o,b_1, b_2\})=\{o\}$ and $p(d(b_1, o,b_2))=b_1$. 
Thus, Lemma~\ref{lemma:ex_citation_to_prop2.6Aic06} yields
\[
o=d(o,o,o)=d(p(b_1), p(o), p(b_2))\mathrel{[\beta, \beta]} p(d(b_1, o, b_2))=b_1.
\]
Hence $(o, b_1)\in \alpha$, in contradiction 
with the choice of $\alpha$ and $\beta$. 
\end{proof}
We now prove a characterization of strictly 1-affine complete algebras 
among those Mal'cev algebras that satisfy (SC1). 
\begin{corollary}\label{cor:MalcevalgebrasSC1_which_one_are_s1ac}
Let $\ab{A}$ be a finite Mal'cev algebra with (SC1). Then the following are equivalent:
\begin{enumerate}
\item $\ab{A}$ is strictly 1-affine complete;\label{item:s1ac:cor:MalcevalgebrasSC1_which_one_are_s1ac}
\item $\ab{A}$ satisfies (AB2).
\label{item:ab2:cor:MalcevalgebrasSC1_which_one_are_s1ac}
\end{enumerate}
\end{corollary}
\begin{proof}
Proposition~\ref{prop:AB2_necessaria} yields that \eqref{item:s1ac:cor:MalcevalgebrasSC1_which_one_are_s1ac} 
implies \eqref{item:ab2:cor:MalcevalgebrasSC1_which_one_are_s1ac}.
Theorem~\ref{teor.main_theorem_sc1_ab2_imply_compl} yields that
\eqref{item:ab2:cor:MalcevalgebrasSC1_which_one_are_s1ac} implies 
\eqref{item:s1ac:cor:MalcevalgebrasSC1_which_one_are_s1ac}. 
\end{proof}

The following generalizes \cite[Proposition~3.3]{AicIdz04}
from finite expanded groups to finite Mal'cev algebras. 
The proof proposed in \cite{AicIdz04} makes no use of the 
fact that $\ab{A}$ is an expanded group. Thus, we omit the proof and 
refer the reader to the proof of~\cite[Proposition~3.3]{AicIdz04}.
\begin{lemma}[{\cite[Proposition~3.3]{AicIdz04}}]\label{lemma:prop3.3AicIdz04implicazione_da_due_a_uno}
Let $\ab{A}$ be a finite Mal'cev algebra 
Then the following are equivalent:
\begin{enumerate}
\item $\ab{A}$ does not satisfy (SC1);\label{item:A_doesn_satisfy_SC1}
\item there exist two join irreducible congruences $\alpha, \beta$ \label{item:the_failure_exists}
such that 
\begin{equation}\label{eq:failure_of_SC1}
[\alpha,\beta]\leq \alpha^-\prec \alpha\leq\beta^-\prec\beta.
\end{equation}
\end{enumerate}
\end{lemma}

\begin{definition}
Let $\ab{A}$ be a finite Mal'cev algebra.
A pair $(\alpha, \beta)$ of join irreducible 
elements of $\Con\ab{A}$ is called 
\emph{a failure of (SC1)} if $\alpha$ and $\beta$
satisfy \eqref{eq:failure_of_SC1}.
\end{definition}
\begin{proposition}\label{prop:condition_on_sc1_failurs}
Let $\ab{A}$ be a finite strictly $1$-affine
complete Mal'cev algebra and let $(\alpha, \beta)$ be 
failure of (SC1). Then $\ab{A}$ satisfies (AB2) and
for all $(a_1, a_2, b)\in A^3$ we have that 
\[
\Cg{\{(a_1, a_2)\}}=\alpha\Rightarrow 
\Cg{\{(b, a_1)\}}\neq \beta.
\]
\end{proposition}
\begin{proof}
Clearly Proposition~\ref{prop:AB2_necessaria}
implies that $\ab{A}$ satisfies (AB2). 
Seeking a contradiction, let us suppose that
$(\alpha, \beta)$ is a failure of (SC1) and
there exist $(o,a, b)\in A^3$ such that
$\Cg{\{(o,a)\}}=\alpha$ and $\beta=\Cg{\{(o,b)\}}$. 
Then 
Lemma~\ref{lemma:ex_citation_to_prop2.6Aic06} implies that
\[
\begin{split}
o&=d(o,b,b)\equiv_{\Cg{\{(b, d(a,o,b)\}}} d(d(o,b,b),d(o,o,b),d(a,o,b))\\
&\equiv_{[\beta, \alpha]}d(d(o,o,a),d(b,o,o),d(b,b,b))\\
&=d(a,b,b)=a.
\end{split}
\]
Moreover, we have $b=d(o,o,b)\mathrel{\alpha}d(a,o,b)$.
Thus, we have 
\[
\alpha=\Cg{\{(o,a)\}}\leq \Cg{\{b, d(a,o,b)\}}\vee [\beta, \alpha]\leq \alpha.
\]
Since $\alpha$ is join irreducible we infer that 
\begin{equation}\label{eq:alpha=congruenzagenerata_d_b}
\alpha=\Cg{\{(b, d(a,o,b)\}}. 
\end{equation}
Moreover, we have 
\[
b=d(o,o,b)\mathrel{\alpha}d(a,o,b)\mathrel{\Cg{\{(o,d(a,o,b))\}}}o.
\]
Thus, since $o=d(a,a,o)\mathrel{\beta} d(a,o,b)$ we have 
\[
\beta=\Cg{\{(o,b)\}}\leq \alpha\vee \Cg{\{(o,d(a,o,b))\}}\leq \beta.
\]
Since $\beta$ is join irreducible, we have  
\begin{equation}\label{eq:beta=congruenzagenerata_d_o}
\beta=\Cg{\{(o, d(a,o,b)\}}. 
\end{equation}
Moreover, we have 
\[
o\mathrel{\alpha} a\mathrel{\Cg{\{(a, d(a,o,b))\}}}d(a,o,b)\mathrel{\alpha} d(a,a,b)=b.
\]
Thus, since $a=d(a,o,o)\mathrel{\beta} d(a,o,b)$ we have 
\[
\beta=\Cg{\{(o,b)\}}\leq \alpha\vee \Cg{\{(a,d(a,o,b))\}}\leq \beta.
\]
Since $\beta$ is join irreducible, we have  
\begin{equation}\label{eq:beta=congruenzagenerata_d_a}
\beta=\Cg{\{(a, d(a,o,b)\}}. 
\end{equation}
Putting together~\eqref{eq:alpha=congruenzagenerata_d_b},
\eqref{eq:beta=congruenzagenerata_d_o},
\eqref{eq:beta=congruenzagenerata_d_a}, 
and \eqref{eq:failure_of_SC1} we infer that
\[
(o,a)\in\Cg{\{(a, d(a,o,b)\}}\cap \Cg{\{(o, d(a,o,b)\}}\cap \Cg{\{(b, d(a,o,b)\}}.
\]
Thus, the partial function $f\colon \{a,o,b,d(a,o,b)\}\to \{a,o\}$ defined by 
$\{a,o,b\}\mapsto o$ and $d(a,o,b)\mapsto a$ is 
congruence preserving. 
Since $\ab{A}$ is strictly 1-affine complete, there exists $p\in\POL\ari{1}\ab{A}$ 
such that for all $x\in\{a,o,b,d(a,o,b)\}$ we have $p(x)=f(x)$.
Thus, Lemma~\ref{lemma:ex_citation_to_prop2.6Aic06} yields 
\[
o=d(o,o,o)=d(p(a),p(o),p(b))\mathrel{[\alpha,\beta]} p(d(a,o,b))=a.
\]
Thus, $(a,o)\in[\alpha, \beta]$, and hence 
$\alpha=\Cg{\{(a,o)\}}\leq [\alpha,\beta]\leq\alpha^-$. Contradiction.
\end{proof}

\begin{definition}
Let $\ab{A}$ be an algebra. We say that $\ab{A}$ is 
\emph{congruence regular} if for all $\theta, \eta\in\Con\ab{A}$ 
and for all $a\in A$ we have 
\[
a/\theta=a/\eta\Leftrightarrow \theta=\eta.
\]
\end{definition}

\begin{proof}[Proof of Theorem~\ref{theorem:Malcev_and_regular_characterization_s1ac}]
The implication ``\eqref{item:ab2and sc1:theorem:Malcev_and_regular_characterization_s1ac}
implies \eqref{item:s1ac:theorem:MMalcev_and_regular_characterization_s1ac}''
follows directly from
Theorem~\ref{teor.main_theorem_sc1_ab2_imply_compl}.
Next, we assume \eqref{item:s1ac:theorem:MMalcev_and_regular_characterization_s1ac} 
and show
\eqref{item:ab2and sc1:theorem:Malcev_and_regular_characterization_s1ac}.
Proposition~\ref{prop:AB2_necessaria} implies that 
$\ab{A}$ satisfies (AB2). 
Moreover,
Proposition~\ref{prop:condition_on_sc1_failurs}
implies that there cannot be a failure of (SC1) in 
a congruence regular Mal'cev algebra that is strictly $1$-affine complete.
In fact if $\alpha$ and $\beta$ are a failure of (SC1), and 
$\alpha=\Cg{\{(a_1, a_2)\}}$, then 
by regularity,
there exists $b\in a_1/\beta\setminus a_1/\alpha$, 
such that the triple $(a_1, a_2, b)\in A^3$ satisfies 
$\alpha=\Cg{\{(a_1, a_2)\}}$ and $\beta=\Cg{\{(a_1, b)\}}$,
in contradiction with 
Proposition~\ref{prop:condition_on_sc1_failurs}. 
Thus, Lemma~\ref{lemma:prop3.3AicIdz04implicazione_da_due_a_uno}
implies that $\ab{A}$ satisfies (SC1),
and
\eqref{item:ab2and sc1:theorem:Malcev_and_regular_characterization_s1ac}
follows.
\end{proof}

In \cite{StaVoj15} the authors developed a theory of commutators
for the 
variety of loops. We build on that to derive a characterization 
of finite strictly 1-affine complete loops
as a corollary of
Theorem~\ref{theorem:Malcev_and_regular_characterization_s1ac}.
As described in detail in \cite[Section~2.4]{StaVoj15},
there exists an order-preserving correspondence between the 
lattice of normal subloops of a loop $\ab{Q}$ and the 
lattice $\Con\ab{Q}$. Following \cite[Section~2.4]{StaVoj15},
we denote the congruence associated to a normal subloop $N$
by $\gamma_N$, and we denote the normal subloop associated 
to a congruence $\alpha$ by $N_\alpha$. 
Given two normal subloops $A$ and $B$ of a loop $\ab{Q}$
we let $[A, B]_{\ab{Q}}:=N_{[\gamma_A, \gamma_B]}$,
and we let $(A:B)_{\ab{Q}}:=N_{(\gamma_A: \gamma_B)}$. 
The center $Z(Q)$ of a loop $\ab{Q}$ is defined by
\[
\begin{split}
&Z(Q):=\\
&\{a\in Q\mid \forall x,y\in Q\colon ax=xa, \, a(xy)=(ax)y, \, x(ya)=(xy)a, \, x(ay)=(xa)y\}.
\end{split}
\]
We start with an elementary lemma on the center of a loop:
\begin{lemma}\label{lemma:all_two_elements_normsubloops_are_contianed_in_the_center}
All the two element normal subloops of a finite loop $\ab{Q}$
are contained in the center. 
\end{lemma}
\begin{proof}
Let $1$ be the identity of $\ab{Q}$ and let
$N=\{1,n\}$ be a two element normal subloop of $\ab{Q}$. 
Clearly,
$(N;\cdot)\cong(\Z_2; +)$. 
Furthermore, the definition of normal subloop implies that 
for all $x,y\in Q$ we have 
\[ xN=Nx, \quad x(yN)=(xy)N,\quad N(xy)=(Nx)y.\]
We first show that
\begin{equation}\label{eq:loop_lemma_xn=nx}
\forall x\in Q \colon xn=nx.
\end{equation} 
Since $N$ is normal 
for each $x\in Q$, we have $xn\in\{nx, x\}$.
Let $x\in Q$ be such that $xn=x$. Then $n$ is the
unique solution of the equation $xz=x$ in the unknown $z$. 
Since $1$ is also a solution, we infer that $1=n$,
contradiction. 

Next, we show that 
\begin{equation}\label{eq:loop_lemma_n(xy)=(nx)y}
\forall x,y\in Q\colon  n(xy)=(nx)y.
\end{equation}
Since $N$ is normal for all $x,y\in Q$, we have 
$n(xy)\in (Nx)y=\{xy, (nx)y\}$. Let $x,y$ be such that
$n(xy)=xy$. Then $n$ is the  unique solution of 
the equation $z(xy)=xy$ in the unknown $z$. 
Since $1$ is also a solution, we infer that $1=n$,
contradiction. 

Next, we show that 
\begin{equation}\label{eq:loop_lemma_x(yn)=(xy)n}
\forall x,y\in Q \colon x(yn)=(xy)n.
\end{equation}
Since $N$ is normal for all $x,y\in Q$, we have 
$(xy)n\in x(yN)=\{xy, x(yn)\}$. Let $x,y$ be such that
$(xy)n=xy$. Then $n$ is the  unique solution of 
the equation $(xy)z=xy$ in the unknown $z$. 
Since $1$ is also a solution, we infer that $1=n$,
contradiction. 

Finally, we show that $(xn)y=x(ny)$ for all $x,y\in Q$.
To this end, let $x,y\in Q$. We have that 
\eqref{eq:loop_lemma_xn=nx} implies
$x(ny)=x(yn)$.
Moreover, by \eqref{eq:loop_lemma_x(yn)=(xy)n}
we have that $x(yn)=(xy)n$. 
Thus, $x(ny)=(xy)n$, and by \eqref{eq:loop_lemma_xn=nx},
we have 
$x(ny)=n(xy)$. Hence \eqref{eq:loop_lemma_n(xy)=(nx)y}
implies that
$x(ny)=(nx)y=(xn)y$.

This concludes the proof that $n\in Z(Q)$. Since $1\in Z(Q)$
we can conclude that $N\subseteq Z(Q)$. 
\end{proof}
\begin{corollary}\label{cor:loops}
For a finite loop $\ab{Q}$ the following are
equivalent:
\begin{enumerate}
\item $\ab{Q}$ is strictly 1-affine complete;\label{item:char_loops_item_affine_complete}
\item there exist $n\in \N$ and a normal subloop $H$ of $\ab{Q}$
such that \label{item:char_loops_item_condition}
\begin{enumerate}
\item $\ab{Q}/H\cong \Z_2^n$, and \label{item:char_loops_subitem_condition_existence}
\item $\forall N\unlhd \ab{Q}\colon 
N\leq H\Rightarrow [N, N]_{\ab{Q}}=N$.\label{item:char_loops_subitem_condition_perfection}
\end{enumerate}
\end{enumerate}
\end{corollary}
\begin{proof}
We first prove that \eqref{item:char_loops_item_affine_complete}
implies \eqref{item:char_loops_item_condition}.
Since $\ab{Q}$ is strictly 1-affine complete,
Theorem~\ref{theorem:Malcev_and_regular_characterization_s1ac}
implies that $\ab{Q}$ satisfies (SC1) and (AB2).
Let $H$ be the intersection of all the normal subloops of
$\ab{Q}$ of index $2$. Then $\ab{Q}/H$ is a subdirect product 
of loops of order $2$. Since a loop of order $2$ is an abelian 
group, $\ab{Q}/H$ is an abelian group.
Moreover, it has exponent $2$, and therefore it is
isomorphic to $\Z_2^n$ for some $n\in \N$.
In order to prove \eqref{item:char_loops_subitem_condition_perfection}
we assume that there exists $B\unlhd \ab{Q}$ such that
$B\leq H$ and $[B, B]_{\ab{Q}}< B$. 
Then there exists a normal subloop $A$ of $\ab{Q}$ such that $A\prec B$
and $[B, B]_{\ab{Q}}\leq A$. 
Let $C$ be a maximal element of the set 
\[\{X \unlhd \ab{Q}\mid A\leq X, \, X\ngeq B\}.\]
Then $C$ is meet irreducible and
$\interval{A}{B}\nearrow \interval{C}{C^+}$.  Thus, 
Lemma~\ref{lemma:ismorphisms_between_classes_same_element_projective_intervals}
and the fact that 
$\ab{Q}$ satisfies (AB2) imply that $C^+$ has exactly two $C$-cosets. 
Let us consider 
$\ab{Q}/C$. Since all the two element normal subloops of $\ab{Q}/C$
are contained in the center
(cf.~Lemma~\ref{lemma:all_two_elements_normsubloops_are_contianed_in_the_center}),
we have that $C^+/C\subseteq Z(Q/C)$. Thus, $\gamma_{C^+/C}
\leq \gamma_{Z(Q/C)}$.
By \cite[Theorem~10.1]{StaVoj15},
$\gamma_{Z(Q/C)}$ is central, and therefore,
$[\uno{Q/C}, \gamma_{Z(Q/C)}]\leq \bottom{Q/C}$. 
By the monotonicity of the commutator
(cf.~\cite[Exercises 4.156(1)]{McKMcnTay88}),
we have that
$[\uno{Q/C}, \gamma_{C^+/C}]\leq \bottom{Q/C}$. 
Hence $[Q, C^+]_{\ab{Q}}\leq C$.
Since $\ab{Q}$ satisfies (SC1), $Q=C^+$.
Thus, $C$ is maximal and therefore, $C\geq H\geq B$,
in contradiction with the choice of $C$. 

Next, we show that  \eqref{item:char_loops_item_condition} implies
(AB2). Let $A$ and $B$ be two normal subloops of $\ab{Q}$
such that $A\prec B$ and 
$[B, B]_\ab{Q}\leq A$. Let us fix two maximal chains in the congruence 
lattice of $\ab{Q}$: 
\[
\begin{split}
\bottom{Q}&=\delta_0\prec\dots \prec\delta_l=\gamma_H\prec\delta_{l+1}\prec\dots\prec\delta_{h}\prec\delta_{h+1}=\uno{Q},\\
\bottom{Q}&=\eta_0\prec\dots \prec\eta_m=\gamma_A\prec\gamma_B=\eta_{m+1}\prec\dots\prec\eta_{h}\prec\eta_{h+1}=\uno{Q}.
\end{split}
\]
Here $h+1$ is the height of $\Con\ab{Q}$. 
Then for all $i\in\{l,\dots, h\}$
each $\delta_{i+1}$-class contains exactly $2$ classes of $\delta_i$. 
Let $\Phi$ be the bijection between the two chains constructed in
the Dedekind-Birkhoff Theorem 
(cf.~\cite[Theorem~2.37]{McKMcnTay88}).
Since the intervals matched by $\Phi$ are projective
and $\interval{\gamma_A}{\gamma_B}$ is abelian,
Lemma~\ref{lemma:exAic18Lemma3.4} implies that
$\Phi$ maps $\interval{\gamma_A}{\gamma_B}$ to an 
interval $\interval{\delta_i}{\delta_{i+1}}$ with $i\geq l$.
Then
Lemma~\ref{lemma:ismorphisms_between_classes_same_element_projective_intervals}
implies that 
each class of $\gamma_B$ contains
exactly $2$ classes of $\gamma_A$. 

Next, we show that  \eqref{item:char_loops_item_condition} implies
(SC1). Seeking a contradiction, let us assume that (SC1) fails. 
Then there exists a meet irreducible 
normal subloop $M$ such that
$(M:M^+)_\ab{Q}> M^+$. If $M\geq H$, 
since the only meet irreducible elements in 
$\interval{H}{Q}$ are the coatoms, we have $M^+=Q$ and therefore 
$(M:M^+)>M^+$. Contradiction.
If $M\ngeq H$, then 
$\interval{M}{M^+}\searrow \interval{M\cap H}{M^+ \cap H}$,
and therefore, Lemma~\ref{lemma:exAic18Lemma3.4}
implies that there exists an abelian prime quotient below $H$
in contradiction with
\eqref{item:char_loops_subitem_condition_perfection}.

Hence we can conclude that 
\eqref{item:char_loops_item_condition} implies (SC1) and (AB2).
Thus, Theorem~\ref{theorem:Malcev_and_regular_characterization_s1ac}
implies  \eqref{item:char_loops_item_affine_complete}.
\end{proof}

\section*{Acknowledgment}
I would like to thank 
E.~Aichinger and P.~Idziak for 
suggesting the proof and content of 
Theorem~\ref{teor:extending_target_minimal-set_to_congruence},
and E.~Aichinger,
P.~Mayr and N.~Mudrinski for several helpful discussions
on commutator theory and the theory of
polynomial interpolation in general algebraic structures. 
Finally, I would like to address a special thank 
to the referee, 
for the
extremely insightful job made in reviewing this paper,
for improving the presentation of the material
and in particular 
for suggesting a less elaborate argument for 
Proposition~\ref{prop:char_congruence_preserving_functions_module},
Theorem~\ref{teor:extending_target_minimal-set_to_congruence} and
Proposition~\ref{prop:AB2_necessaria}.
\bibliographystyle{abbrv.bst}
\providecommand*{\url}[1]{\texttt{\detokenize{#1}}}
  \ifx\SetBibliographyCyrillicFontfamily\undefined\def\SetBibliographyCyrillicFontfamily{\relax}\fi
  \def\Cyr#1{\bgroup\SetBibliographyCyrillicFontfamily\fontencoding{T2A}\selectfont{#1}\egroup}
  \def\Palatalization#1{\bgroup\fontencoding{T1}\selectfont\v{#1}\egroup}

\end{document}